\documentclass[12pt]{iopart}

\usepackage[utf8]{inputenc}
\usepackage{tikz}
\usepackage{subfig}
\usepackage{mathrsfs}
\usepackage[osf,sc]{mathpazo}
\usepackage{microtype}
\usepackage{courier}
\usepackage[utf8]{inputenc}
\usepackage[T1]{fontenc}
\usepackage[numbers,sort,compress]{natbib}
\usepackage{graphicx}
\usepackage{hyperref}
\usepackage{doi}
\hypersetup{
   unicode=true,
   pdftoolbar=false,
   pdfmenubar=false,
   pdffitwindow=false,
   pdfnewwindow=true,
   pdfkeywords={},
   colorlinks=true,
   linkcolor=black,
   citecolor=black,
   filecolor=black,
   urlcolor=black
}
\usepackage{amsthm}
\usepackage{thmtools}
\declaretheoremstyle[headfont=\normalfont, bodyfont=\itshape, postheadspace=1em]{assumpstyle}
\declaretheorem[style=assumpstyle]{condition}

\newtheorem{algorithm}{Algorithm}[section]
\newtheorem{theorem}[algorithm]{Theorem}
\newtheorem{lemma}[algorithm]{Lemma}
\newtheorem{proposition}[algorithm]{Proposition}
\newtheorem{corollary}[algorithm]{Corollary}
\newtheorem{remark}[algorithm]{Remark}

\newcommand{\R}{\mathbb{R}}

\newcommand{\X}{\mathbf{X}}
\newcommand{\Y}{\mathbf{Y}}
\newcommand{\F}{\mathcal{F}}
\newcommand{\Fls}{\mathcal{F}^{i_{s}^{\ell}}}
\newcommand{\vets}{\mathbf{s}}
\newcommand{\vetr}{\mathbf{r}}
\newcommand{\vetx}{\mathbf{x}}
\newcommand{\vety}{\mathbf{y}}
\newcommand{\vetz}{\mathbf{z}}
\newcommand{\vetv}{\mathbf{v}}
\newcommand{\vetg}{\mathbf{g}}

\newcommand{\veth}{\mathbf{h}}

\newcommand{\imb}{\mathbf{b}}
\newcommand{\bb}{\begin{equation}}
\newcommand{\ee}{\end{equation}}
\newcommand{\ds}{\displaystyle}
\newcommand{\II}{\mathcal{I}}

\newcommand{\PX}{\mathcal{P}_{\mathbf{X}}}
\newcommand{\OO}{\mathcal{O}}
\newcommand{\V}{\mathcal{V}}
\newcommand{\SSS}{\mathcal{S}}

\begin{document}
 
\title[]{String-Averaging Incremental Subgradients for Constrained Convex Optimization with Applications to Reconstruction of 
Tomographic Images}
\author{Rafael Massambone de Oliveira$^1$, Elias Salom\~ao Helou$^1$ and Eduardo Fontoura Costa$^1$}
\address{$^1$ University of S\~ao Paulo - Institute of Mathematics and Computer Sciences, Department of Applied Mathematics and Statistics, S\~ao Carlos-SP, CEP 13566-590, Brazil}
\ead{massambone@usp.br, elias@icmc.usp.br, efcosta@icmc.usp.br}

\begin{abstract}
We present a method for non-smooth convex minimization which is based on subgradient directions and string-averaging techniques. 
In this approach, the set of available data is split into sequences (strings) and a given iterate is processed independently 
along each string, possibly in parallel, by an incremental subgradient method (ISM). 
The end-points of all strings are averaged to form the next iterate. 
The method is useful to solve sparse and large-scale non-smooth convex optimization problems, such as those arising in 
tomographic imaging. 
A convergence analysis is provided under realistic, standard conditions. 
Numerical tests are performed in a tomographic image reconstruction application, showing good performance for the 
convergence speed when measured as the decrease ratio of the objective function, in comparison to classical ISM.
\end{abstract}
\vspace{2pc}\noindent{\it Keywords:} convex optimization, incremental algorithms, subgradient methods, projection methods, string-averaging algorithms.
\maketitle

\section{Introduction} \label{intro}

A fruitful approach to solve an inverse problem is to recast it as an optimization
problem, leading to a more flexible formulation that can be handled with different
techniques. The \textit{reconstruction of tomographic images} is a classical example of a
problem that has been explored by optimization methods, among which the well-known 
\textit{incremental subgradient method} (ISM) \cite{solodov98, nedic01, helou09}, that is 
a variation of the \textit{subgradient method} \cite{dev85, shor85, polyak87}, features 
nice performance in terms of convergence speed.
There are many papers that discuss incremental gradient/subgradient algorithms for convex/non-convex objective 
functions (smooth or not) with applications to several fields \cite{bertsekas1997new,bertsekas2000gradient,
blatt2007convergent,nedic01,sundhar2010distributed,sundhar2009incremental,solodov1998incremental,solodov98,
tanaka2003subset,tseng1998incremental}. Some examples of applications to tomographic 
image reconstruction are found in \cite{ahn2003globally,browne1996row,de2001fast,hudson1994accelerated,
tanaka2003subset}.
In this paper, we consider a rather general optimization
problem that can be addressed by ISM and is useful for tomographic reconstruction
and other problems, including to find solutions for ill-conditioned and/or large-scale linear systems. This problem
consists of determining:
\bb \label{1-intro} \begin{array}{c} \ds \vetx \in \arg \min f(\vetx)
\\
\mbox{s.t.} \quad \vetx \in \X \subset \R^n, \end{array} \ee
where:
\begin{description}
\item[(i)] $f(\vetx) := P\sum_{i=1}^{m} \xi_i f_{i}(\vetx)$ in which $f_i : \R^n \rightarrow \R$ 
are convex (and possibly non-differentiable) functions;
\item[(ii)] $\X$ is a non-empty, convex and closed set;
\item[(iii)] The set $\II = \{ S_1, \dots, S_P \}$ is a \textit{partition} of $\{1, \dots, m \}$, i.e., 
$S_{\ell} \cap S_{j} = \emptyset$ for any $\ell, j \in \{ 1, \dots, P \}$ with $\ell \neq j$ 
and $\bigcup_{\ell=1}^{P} S_{\ell} = \left\{1, \dots, m\right\}$;
\item[(iv)] Given $w_\ell\in [0,1]$ and $S_\ell \in \II$, $\ell=1,\dots,P$, the weights $\xi_i$ satisfy $\xi_i = w_\ell$ for
all $i \in S_\ell$;
\item[(v)] $\sum_{\ell=1}^{P} w_{\ell} = 1$.
\end{description}

Problem (\ref{1-intro}) with conditions \textbf{(i)}-\textbf{(v)} is reduced to the classical problem of minimizing 
$\sum_{i=1}^{m} f_i(\vetx)$, s.t. $\vetx \in \X$, when $w_\ell = 1/P$ for all $\ell = 1, \dots, P$. The reason why we 
write the problem in this more complex way is twofold.
On one hand, it is common to find problems in which a set of fixed weights are used to 
prioritize the contribution of some component functions. For instance, in the context of 
distributed networks, the component functions $f_{i}$ (also called ``agents'') can be affected by external conditions, 
network topology, traffic, etc., making possible that some sets of agents have 
a prevalent role on the network, which can be modeled
by the weights (related to the corresponding subnets).
On the other hand, the weights and the partition, which bring flexibility to the model
and could possibly be explored aiming for instance at faster convergence, will fit naturally
in our algorithmic framework.

We consider an approach that mixes ISM and \textit{string-averaging algorithms} (SA algorithms). The general form of the SA 
algorithm was proposed initially in \cite{censor01} and applied in solving 
\textit{convex feasibility problems} (CFPs) with algorithms that use {\it projection methods} \cite{censor01,censor03,penfold10}. 
Strings are created so that 
ISM (more generally, any $\epsilon$-incremental subgradient method) can be processed in an 
independent form for each string (by step operators). Then, an average of string iterations is computed 
(combination operator), guiding the iterations towards the solution. To complete, approximate projections 
are used to maintain feasibility. We provide an analysis of convergence under reasonable assumptions, such as diminishing step-size rule 
and boundedness of the subgradients.

Some previous works in the literature have improved the understanding and practical efficiency of ISM by 
creating more general algorithmic structures, enabling a broader analysis of convergence and making them more robust 
and accurate \cite{nedic01, helou09, sundhar2009incremental, sundhar2010distributed, johansson2010}. 
We improve on those results by adding a string-averaging structure to the ISM that allows for an efficient 
parallel processing of a complete iteration which, consequently, can lead to fast convergence and suitable approximate solutions. 
Furthermore, the presented techniques present better smoothing properties in practice, which is good for imaging tasks. 
These features are desirable, especially when we seek to solve ill-conditioned/large scale problems. As mentioned at the 
beginning of this section, one of our goals is to obtain an efficient method for solving problems of reconstruction of 
tomographic images from incompletely sampled data.

Although our work is closely linked to ISM, it is important to mention other classes of methods that can be applied 
to convex optimization problems.
Under reasonable assumptions, problem (\ref{1-intro}) can be 
solved using \textit{proximal-type algorithms} (for a description of some of the main methods, see \cite{combettes2011}). 
For instance, in \cite{combettes2008} the authors propose a proximal decomposition method derived from the Douglas-Rachford algorithm and establish
its weak convergence in Hilbert spaces. Some variants and generalizations of such methods can be found in \cite{bredies2009, 
raguet2013,combettes2016, chen1997, combettes2011}. 
The \textit{bundle approach}~\cite{hil93} is often used for numerically solving non-differentiable convex optimization problems. 
Also, \textit{first order accelerated techniques}~\cite{bet09,bet09b} form yet another family of popular techniques for convex optimization problem endowed with a certain separability property. 
The advantage of ISM over the aforementioned techniques lies in its lightweight iterations from the computational viewpoint, 
not even requiring sufficient decrease properties to be checked. Besides, ISM presents a fast practical convergence rate in 
the first iterations which enables this 
technique to achieve good reconstructions within a small amount of time even for the huge problem sizes that appear, for example, in tomography.

The {\it tomographic image reconstruction problem} consists in finding an image described by a function 
$\psi : \mathbb R^2 \to \mathbb R$ 
from samples of its {\it Radon Transform}, which can be recast into solving the linear system
\begin{equation}
   \label{image system} R \vetx = \imb,
\end{equation}
where $R$ is the discretization of the Radon Transform operator, $\imb$ contains the collected samples of the 
Radon Transform and the desired image is represented by the vector $\vetx$. We consider solving the problem (\ref{image system}), 
rewriting it as a minimization problem, as in (\ref{1-intro}).
Solving problem (\ref{image system}) from an optimization standpoint is not a new idea. In particular, \cite{helou09} 
illustrates the application of some 
of the methods arising from a general framework to tomographic image reconstruction with a limited number of 
views. For the discretized problem (\ref{image system}), iterative methods such as ART (Algebraic Reconstruction Technique) 
\cite{natterer2001mathematical}, POCS (Projection Onto Convex Sets) \cite{bauschke1996projection, combettes97}, and 
Cimmino \cite{combettes1994inconsistent} have been widely used in the past.

Tomographic image reconstruction is an inverse problem in the sense that the image $\psi$ is to be obtained from the 
inversion of a linear compact operator, which is well known to be an ill-conditioned problem. While the specific case of 
Radon inversion has an analytical solution, that was published in 1917 by Johann Radon 
(for details see \cite{natterer86}), both such analytical techniques and the aforementioned iterative methods for linear systems of equations
suffer from amplification of the statistical noise which, in practice, is always present  in the right-hand side of~(\ref{image system}).
Therefore, methods designed to deal with noisy data have been developed, based on a maximum likelihood approach, 
among which EM (Expectation Ma\-xi\-mi\-za\-ti\-on) \cite{shepp1982maximum, vardi1985statistical}, 
OS-EM (Ordered Subsets Expectation Ma\-xi\-mi\-za\-ti\-on) \cite{hudson1994accelerated}, 
RAMLA (Row-Action Maximum Likelihood Algorithm) \cite{browne1996row}, 
BSREM (Block Sequential Regularized Expectation Ma\-xi\-mi\-za\-ti\-on) \cite{de2001fast}, 
DRAMA (Dynamic RAMLA) \cite{tanaka2003subset, helou05}, modified BSREM and relaxed 
OS-SPS (Ordered Subset-Se\-pa\-ra\-ble Paraboloidal Surrogates) \cite{ahn2003globally} are some of the best known in 
the literature. 
In \cite{hcc14}, a variant 
of the EM algorithm was introduced, called \textit{String-Averaging Expectation-Maximization} (SAEM) \textit{algorithm}. 
The SAEM algorithm 
was used in problems of image reconstruction in {\it Single-Photon Emission Computerized To\-mo\-gra\-phy} (SPECT) and 
showed good performance in simulated and real data studies. High-contrast images, with less noise and clearer object 
boundaries were reconstructed without incurring in more computation time.
Besides the BSREM, DRAMA, modified BSREM and relaxed OS-SPS, that are relaxed 
algorithms for (penalized) maximum-likelihood image reconstruction in tomography, the method introduced in \cite{dewaraja2010} considers 
an approach, based in OS-SPS, in which extra anatomical boundary information is
used. Other methods that use penalized models can be found in \cite{harmany2012, chouzenoux2013}. Proximal 
methods were used in \cite{anthoine2012} to reconstruct images obtained via {\it Cone Beam Computerized Tomography} (CBCT) and 
{\it Positron Emission Tomography} (PET). In \cite{chouzenoux2013}, the {\it Majorize-Minimize Memory Gradient algorithm} 
\cite{chouzenoux2011, chouzenoux2013siam} is studied and applied to imaging tasks.

The paper is organized as follows:
Section \ref{sec.2} contains some preliminary theory involving incremental subgradient methods, optimality and feasibility 
operators and string-averaging algorithm; 
Section \ref{sec.3} discusses the proposed algorithm to solve (\ref{1-intro}), \textbf{(i)}-\textbf{(v)}; 
Section \ref{sec.4} shows theoretical convergence results; 
in Section \ref{sec.5} numerical tests are performed with reconstruction of tomographic images. 
Final considerations are given in Section \ref{sec.6}.

\section{Preliminary theory} \label{sec.2}

Throughout the text, we will use the following notations: bold-type notations e.g. $\vetx$, $\vetx_i$ and $\vetx_{i}^{k}$ are vectors 
whereas $x$ is a number. We denote $x_i$ as the $i$th coordinate of vector $\vetx$. Moreover,
\begin{eqnarray*} \ds \PX(\vetx):= \arg \min_{\vety \in \X} \left\| \vety - \vetx \right\|, \quad d_{\X}(\vetx):= \left\|\vetx - \PX(\vetx) \right\|, \\ 
\left[ x \right]_{+} := \max\left\{0,x\right\}, \quad f^{\ast} = \inf_{\vetx \in \X} f(\vetx)  \quad \mbox{and} \quad \X^{\ast} = \left\{\vetx \in \X \, | \, f(\vetx) = f^{\ast}\right\},\end{eqnarray*}
where we assume that $\X^{\ast} \neq \emptyset$.

One of the main methods for solving (\ref{1-intro}) is the
\textit{subgradient method}, whose extensive theory can be found in \cite{dev85, shor85, polyak87, hil93, bertsekas1999nonlinear},
\bb \label{subgrad-method} \vetx^{k+1} = \PX\left(\vetx^k - \lambda_k \sum_{i=1}^{m} \vetg_{i}^{k}\right), \quad \lambda_k > 0, \quad \vetg_{i}^{k} \in \partial f_i(\vetx^k),\ee 
where the {\it subdifferential} of $f: \R^n \rightarrow \R$ at $\vetx$ (the set of all subgradients) 
can be defined by
\bb \ds \partial f(\vetx) := \left\{\vetg \, | \, f(\vetx) + \left\langle \vetg, \vetz - \vetx \right\rangle \leq f(\vetz), \,\, \forall \vetz\right\}. \ee
A similar approach to (\ref{subgrad-method}), known as \textit{incremental subgradient method}, was studied 
firstly by Kibardin in \cite{kibardin1979decomposition} and then analyzed by Solodov and 
Zavriev in \cite{solodov98}, in which a complete iteration of the algorithm can be described as follows:
\begin{eqnarray} \vspace{0.3cm} \vetx_{0}^{k} &=& \vetx^{k} \nonumber
 \\ \label{sub-iterate} \vspace{0.3cm} \vetx_{i}^{k} &=& \vetx_{i-1}^{k} - \lambda_k \vetg_{i}^{k}, \quad i = 1, \dots, m, \quad \vetg_{i}^{k} \in \partial f_{i}(\vetx_{i-1}^{k}) 
 \\ \ds \vetx^{k+1} &=& \PX\left(\vetx_{m}^{k}\right). \nonumber \end{eqnarray}
A variant of this algorithm that uses projection onto $\X$ to compute the sub-iterations $\vetx_{i}^{k}$ was analyzed in \cite{nedic01}.

The method we propose in this paper for solving the problem given in (\ref{1-intro}), \textbf{(i)}-\textbf{(v)} has the 
following general form described in \cite{helou09}:
\bb \label{ov} \begin{array}{rcl} \vetx^{k+1/2} &=& \ds \OO_f(\lambda_k, \vetx^k); \\ \vetx^{k+1} &=& \ds \V_{\X}(\vetx^{k+1/2}). \end{array} \ee
In the above equations, $\OO_f$ is called {\it optimality operator} and $\V_{\X}$ is the 
{\it feasibility operator}.
This framework was created to handle quite general algorithms for convex optimization problems. 
The basic idea consists in dividing an iterate in two parts: an optimality step which tries to guide the iterate towards the 
minimizer of the objective function (but not necessarily in a descent direction), followed by the feasibility step that drives 
the iterate in the direction of feasibility.

Next we enunciate a result due to Helou and De Pierro (see \cite[Theorem 2.5]{helou09}), establishing convergence of the method 
(\ref{ov}) under some conditions. This result is the key for the convergence analysis of the algorithm we propose in section 
\ref{sec.3}.

\begin{theorem} \label{conv1} The sequence $\{\vetx^k \}$ generated by the method described in \emph{(\ref{ov})} converges in the sense that
\begin{equation*} \ds d_{\X^{\ast}}(\vetx^k) \rightarrow 0 \qquad \mbox{and} \qquad \lim_{k \rightarrow \infty} f(\vetx^k) = f^{\ast}, \end{equation*}
if all of the following conditions hold:
\begin{condition}[Properties of optimality operator] \label{cond1} For every $\vetx \in \X$ and for all sequence $\lambda_k \geq 0$, there exist 
$\alpha > 0$ and a sequence 
$\rho_k \geq 0$ such that the optimality operator $\OO_f$ satisfies for all $k \geq 0$
\bb \label{cond1-eq1} \ds \|\OO_f(\lambda_k, \vetx^k) - \vetx \|^2 \leq \|\vetx^k - \vetx \|^2 - \alpha \lambda_k (f(\vetx^k) - f(\vetx)) + \lambda_k \rho_k. \ee
We further assume that the error term in the above inequality vanishes, i.e., $\rho_k \to 0$ and we consider a 
boundedness property for the optimality operator: there is $\gamma > 0$ such that
\bb \label{cond1-eq2} \ds \| \vetx^k - \OO_f (\lambda_k, \vetx^k) \| \leq \lambda_k \gamma. \ee
\end{condition}
\begin{condition}[Property of feasibility operator] \label{cond2} For the feasibility operator $\V_{\X}$, we impose that for all 
$\delta > 0$, exists $\epsilon_{\delta} > 0$ such that, if $d_{\X}(\vetx^{k+1/2}) \geq \delta$ and $\vetx \in \X$ we have
\bb \ds \|\V_{\X}(\vetx^{k+1/2}) - \vetx \|^2 \leq \| \vetx^{k+1/2} - \vetx \|^2 - \epsilon_{\delta}. \ee
Moreover, for all $\vetx \in \X$, $\V_{\X}(\vetx) = \vetx$, i.e., $\vetx$ is a fixed point of $\V_{\X}$.
\end{condition} 
\begin{condition}[Diminishing step-size rule] \label{cond3} The sequence $\{ \lambda_k \}$ satisfies
\bb \label{hippasso} \ds \lambda_k \rightarrow 0^{+}, \qquad \sum_{k=0}^{\infty} \lambda_k = \infty.\ee
\end{condition}
\begin{condition} \label{cond4} The optimal set $\X^{\ast}$ is bounded, $\{ d_{\X}(\vetx^k) \}$ is bounded and 
\begin{equation*} [f(\PX(\vetx^k)) - f(\vetx^k)]_{+} \rightarrow 0. \end{equation*}
\end{condition}
\end{theorem}

\begin{remark} \label{rmk1} \normalfont
Regarding the requirement $[f(\PX(\vetx^k)) - f(\vetx^k)]_{+} \rightarrow 0$, it holds if there is a bounded sequence $\{ \vetv^k \}$ where 
$\vetv^k \in \partial f(\PX(\vetx^k))$ and $d_{\X}(\vetx^k) \to 0$. Indeed, 
\begin{equation*}
\langle \vetv^k, \vety - \PX(\vetx^k) \rangle \leq f(\vety) - f(\PX(\vetx^k)), \quad \forall \, \vety \in \mathbb{R}^n.
\end{equation*}
By Cauchy-Schwarz inequality, we have $\| \vetv^k \| \|\PX(\vetx^k) - \vety \| \geq [f(\PX(\vetx^k)) - f(\vety)]_{+}$. Taking $\vety = 
\vetx^k$, then $d_{\X}(\vetx^k) \to 0$ ensures that $[f(\PX(\vetx^k)) - f(\vetx^k)]_{+} \to 0$. Therefore, under this mild boundedness assumption 
on the subdifferentials $\partial f(\PX(\vetx^k))$, proving that $d_{\X}(\vetx^k) \to 0$ also ensures that $[f(\PX(\vetx^k)) - f(\vetx^k)]_{+} \rightarrow 0$.

Concerning the assumption $d_{\X}(\vetx^k) \to 0$, Proposition 2.1 in \cite{helou09} shows that it holds if $\{d_{\X}(\vetx^k) \}$ is bounded, 
$\lambda_k \to 0^{+}$, and Equation 
(\ref{cond1-eq2}) plus Condition \ref{cond2} hold. Since Condition \ref{cond4} requires $\{d_{\X}(\vetx^k) \}$ to be bounded, then we have that 
$[f(\PX(\vetx^k)) - f(\vetx^k)]_{+} \to 0$ just under the boundedness assumption on $\partial f(\PX(\vetx^k))$. Furthermore, Corollary 2.7 in \cite{helou09} states 
that $d_{\X}(\vetx^k) \to 0$ if $\lambda_k \to 0^{+}$, Conditions \ref{cond1} and \ref{cond2} hold and there is $f_l$ such that $f(\vetx^k) \geq f_l$ for all 
$k$. Basically, the hypotheses of this corollary allow to show that $\{d_{\X}(\vetx^k) \}$ is bounded and result follow by Proposition 2.1. This is the situation that occurs 
in our numerical experiment. Such remarks are important to show how the hypotheses of our main convergence result (see Corollary \ref{convSA} in section 
\ref{sec.4}) can be reasonable. \qed
\end{remark}

To state our algorithm in next section, we need to define the operators $\OO_f$ and $\V_{\X}$. Below we present the 
last ingredient of our operator $\OO_f$, the \textit{String-Averaging} (SA) \textit{algorithm}.
Originally formulated in \cite{censor01}, SA algorithm consists of dividing an index set $\mathrm{I} = \left\{1,2, \dots, \eta \right\}$ 
into {\it strings} in the following manner
\bb \ds \Delta_{\ell} := \left\{i_{1}^{\ell}, i_{2}^{\ell}, \dots, i_{m(\ell)}^{\ell}\right\}, \ee
where $m(\ell)$ represents the number of elements in the string $\Delta_{\ell}$ and $\ell \in \{ 1,2, \dots,  N \}$. 
Let us consider $\mathcal{X}$ and $\mathcal{Y}$ as subsets of $\R^{n}$ where $\mathcal{Y} \subseteq \mathcal{X}$. 
The basic idea behind the method consists in the sequential application of 
{\it step operators} $\Fls: \mathcal{X} \to \mathcal{Y}$, for each $s = 1,2, \dots, m(\ell)$
over each string $\Delta_{\ell}$, producing $N$ vectors $\vety_{\ell}^{k} \in \mathcal{Y}$. Next, a {\it combination operator} 
$\F: \mathcal{Y}^N \to \mathcal{Y}$ mixes, usually by weighted average, all vectors $\vety_{\ell}^{k}$ to obtain 
$\vety^{k+1}$. We refer to the index $s$ as the \emph{step} and the index $k$ as the \emph{iteration}.
Therefore, given $\vetx^0 \in \mathcal{X}$ and strings $\Delta_1, \dots, \Delta_N$ of $\mathrm{I}$, 
a complete iteration of the SA algorithm is 
computed, for each $k \geq 0$, by equations
\bb \label{step_op} \ds \vety_{\ell}^{k} := \F^{i_{m(\ell)}^{\ell}} \circ \dots \circ \F^{i_{2}^{\ell}} \circ \F^{i_{1}^{\ell}}\left(\vetx^k\right), \ee
\bb \label{comb_op} \ds \vety^{k+1} := \F((\vety_{1}^{k}, \dots, \vety_{N}^{k})). \ee
The main advantage of this approach is to allow for computation of each vector $\vety_{\ell}^{k}$ \emph{in parallel} at each 
iteration $k$, which is possible because the step operators $\F^{i_{1}^{\ell}}, \dots, \F^{i_{m(\ell)}^{\ell}}$ act along each 
string independently.

\section{Proposed algorithm} \label{sec.3}

Now we are ready to define $\OO_f$ and $\V_{\X}$. Let us start by defining the optimality operator 
$\OO_f: \R_{+} \times \Y \to \Y$, where $\Y$ is a non-empty, closed and convex set 
such that $\X \subset \Y \subseteq \R^{n}$. For this, let $ \F^{i_{s}^{\ell}}: \R_{+} \times \Y \to \Y$ and $\F: \Y^P \to \Y$. Consider 
the set of strings $\Delta_1 = S_1, \dots, \Delta_P = S_P$ 
and the weight set $\{ w_\ell \}_{\ell=1}^{P}$ as defined in the problem given in (\ref{1-intro}) with conditions 
\textbf{(iii)}-\textbf{(v)}.  
Then, given $\vetx \in \Y$ and $\lambda \in \R_{+}$, we define
\bb \label{initial_opt_op} \vetx_{i_{0}^{\ell}} := \vetx, \quad \mbox{for all} \quad \ell = 1, \dots, P,\ee
\bb \label{step.operators} \ds \vetx_{i_{s}^{\ell}} := \F^{i_{s}^{\ell}}(\lambda, \vetx_{i_{s-1}^{\ell}}) := \vetx_{i_{s-1}^{\ell}} - 
\lambda \vetg_{i_{s}^{\ell}}, \quad s = 1, \dots, m(\ell), \ee
\bb \label{end.points} \vetx_{\ell} := \vetx_{i_{m(\ell)}^{\ell}}, \quad \ell = 1, \dots, P, \ee
\bb \label{SA_isoo} \OO_f(\lambda, \vetx) := \F((\vetx_1, \dots, \vetx_P)) := \sum_{\ell=1}^{P} w_{\ell} \vetx_{\ell}, \ee
where $\vetg_{i_{s}^{\ell}} \in \partial f_{i_{s}^{\ell}}(\vetx_{i_{s-1}^{\ell}})$.
Operators $\F^{i_{s}^{\ell}}$ in (\ref{step.operators}) correspond to the step operators in equation (\ref{step_op}) of the SA algorithm and its 
definition is motivated by equation 
(\ref{sub-iterate}) of the incremental subgradient method. Function $\F$ in (\ref{SA_isoo}) corresponds to the combination 
operator in (\ref{comb_op}) and performs a weighted average 
of the end-points $\vetx_{\ell}$, completing the definition of the operator $\OO_f$.

Now we need to define a feasibility operator $\V_{\X}$. 
For that, we use the {\it subgradient projection} \cite{bauschke06,combettes97,yamada2005adaptive,yamada2005hybrid}. 
Let us start noticing that every convex set $\X \neq \emptyset$ can be written as
\bb \label{feasible.set} \ds \X = \bigcap_{i=1}^{t} \, lev_0(h_i),\ee
where $lev_0(h_i) := \left\{\vetx \,| \, h_{i}(\vetx) \leq 0\right\}$. Each function $h_i : \R^n \rightarrow \R$ ($t$ is finite) is supposed to be convex. 
The feasibility operator $\V_{\X}: \R^n \rightarrow \R^n$ is defined in \cite{helou09} in the following form:
\bb \label{viabilidade} \ds \V_{\X}:= \SSS_{h_t}^{\nu_t} \circ \SSS_{h_{t-1}}^{\nu_{t-1}} \circ \dots \circ \SSS_{h_1}^{\nu_1}. \ee
This definition assumes that there is $\sigma \in (0,1]$ such that $\nu_i \in [\sigma, 2 - \sigma]$ for all $i$. Each operator 
$\SSS_{h}^{\nu}: \R^n \rightarrow \R^n$ in the previous definition is constructed using a $\nu$-relaxed version of the 
subgradient projection with Polyak-type step-sizes , i.e.,
\bb \label{subgrad_projection} \ds \SSS_{h}^{\nu}(\vetx) := \left\{ \begin{array}{ccl} \vetx - \nu \frac{[h(\vetx)]_{+}}{\left\|\veth\right\|^2} \veth, & \mbox{if} & \ds \veth \neq \mathbf{0}; \\ \vetx, &  & \mbox{otherwise}, \end{array} \right. \ee
where $\nu \in (0,2)$ and $\veth \in \partial h(\vetx)$.

In order to get a better understanding of the behavior of our feasibility operator, Figure \ref{feasibility_ex} 
shows the trajectory taken by successive applications of the operator $\V_{\X}$. The feasible
set $\X$ is the intersection of the zero sublevel sets of the following convex functions: $h_1(\vetx) = \langle \mathbf{a}, 
\vetx \rangle + 2\| \vetx \|_1 - 1$, $h_2(\vetx) = 3 \| \vetx \|_{\infty} - 2.5$ and 
$h_3(\vetx) = \| A \vetx - \mathbf{a} \|_1 + 2\|B \vetx - \mathbf{c} \|_2 - 10$ where 
\begin{equation*}
A = \left[ \begin{array}{cc} 2 & 1 \\ -1 & 3 \end{array} \right], \quad
B = \left[ \begin{array}{cc} 1 & 0 \\ -2 & 2 \end{array} \right], \quad 
\mathbf{a} = [2 \quad 1]^T \quad \mbox{and} \quad
\mathbf{c} = [1 \,\, -2]^T.
\end{equation*}
To obtain $\V_{\X}(\vetx) = \SSS_{\veth_3}^{\nu_3} \circ \SSS_{\veth_2}^{\nu_2} \circ \SSS_{\veth_1}^{\nu_1}(\vetx)$, 
we compute the subgradients $\veth_i \in \partial h_i(\vets_{i-1})$, $i = 1,\dots,3$, such that $\vets_0 := \vetx$ and 
$\vets_i := \SSS_{\veth_i}^{\nu_i}(\vets_{i-1})$. We choose $[-3 \,\, -2.5]^T$ as an initial point and the following relaxation parameters: 
$\nu_1 = 0.5$, $\nu_2 = 0.6$ and $\nu_3 = 0.7$.
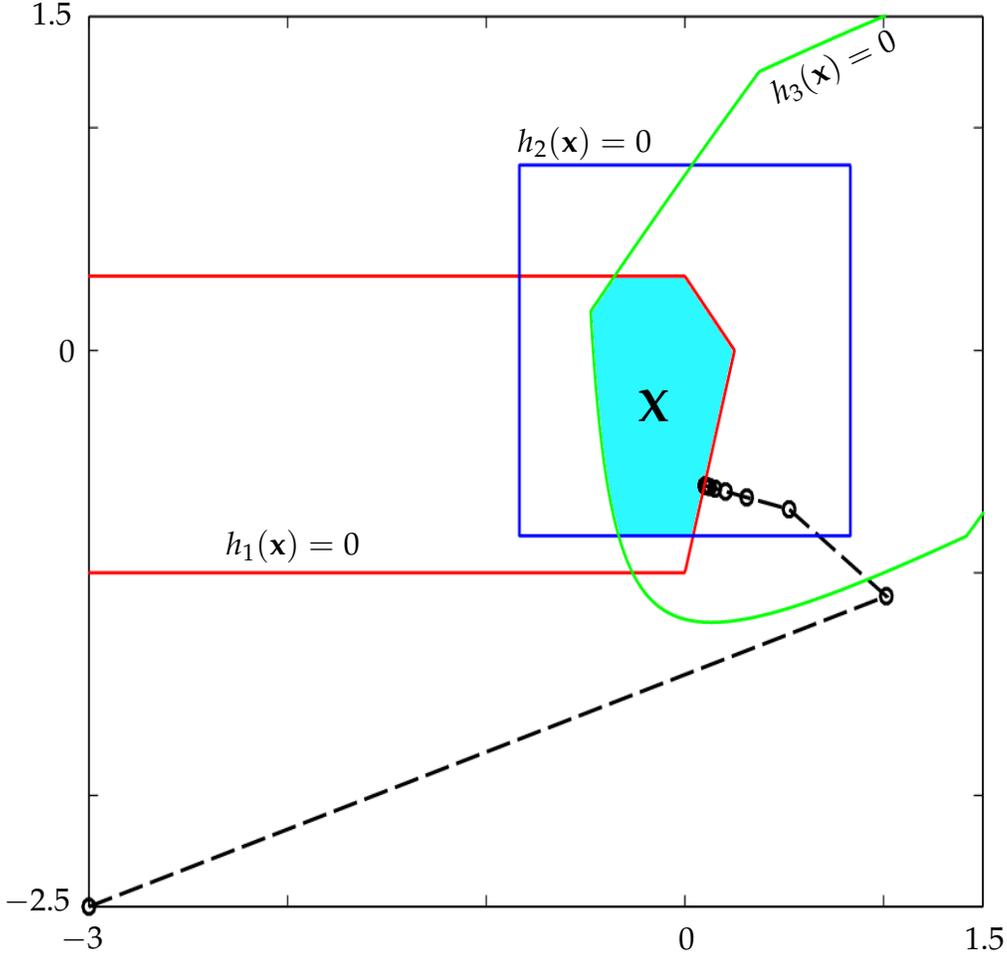
\begin{figure}
\centering
\begin{tikzpicture}[scale=0.4]
\node (label) at (0,0){
\includegraphics[width=12cm,height=12cm]{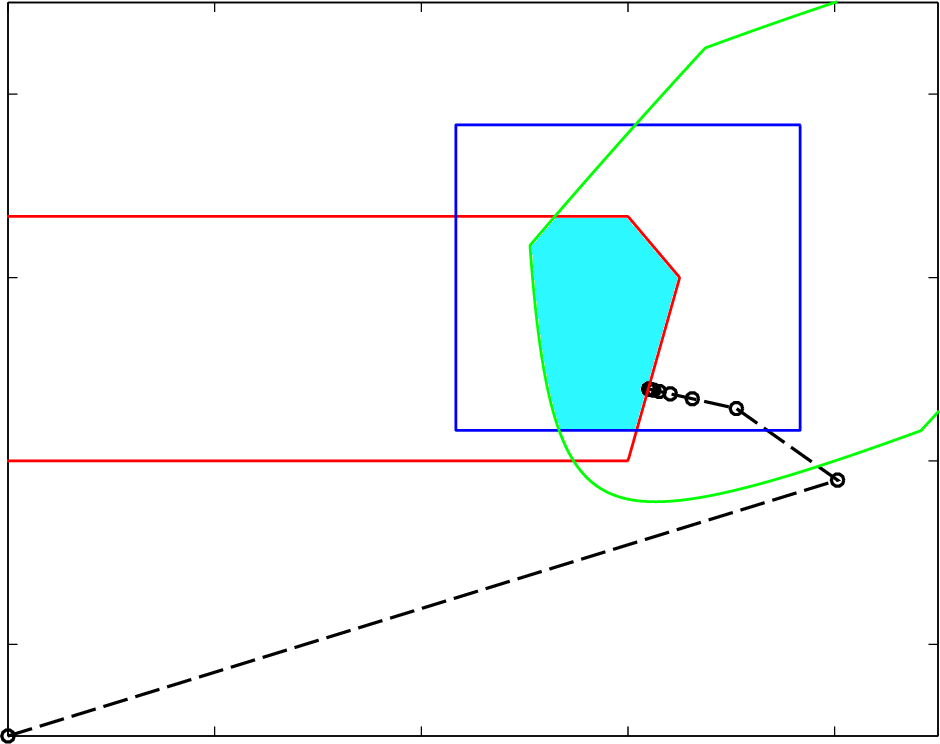}
};
\node  (h1) at (-8,-2.7) {$h_1(\vetx) = 0$};
\node  (h2) at (1.7,10.7) {$h_2(\vetx) = 0$};
\node[rotate = 25]  (h3) at (10,13.2) {$h_3(\vetx) = 0$};
\node  (X) at (4,2) {\Large $\X$};
\node at (15,-15.75) {$1.5$};
\node at (5.1,-15.75) {$0$};
\node at (-15,-15.75) {$-3$};
\node at (-16,15) {$1.5$};
\node at (-15.5,3.8) {$0$};
\node at (-16.5,-14.5) {$-2.5$};
\end{tikzpicture}
\caption{Ten successive applications of the feasibility operator $\V_{\X}$. The circles represent the points obtained in each application, 
starting from the point $[-3 \,\, -2.5]^T$. The dashed line describes the trajectory formed by these points.} \label{feasibility_ex}
\end{figure}

\begin{remark} \normalfont
A string-averaging version of the feasibility operator can easily be derived
in the following manner. Consider $Q$ strings 
$V_j := \{ i^j_1, i^j_2, \dots, i^j_{\kappa( j )} \} \subset \left\{1, \dots, t\right\}$ such that $\bigcup_{j=1}^{Q} V_j = \left\{1, \dots, t\right\}$, 
where $\kappa(j)$ is the number of elements in the string $V_j$. Then, for each $j = 1, \dots, Q$, we define the \emph{string feasibility operator} $\V_j$ as
\begin{equation*}
\mathcal V_j := \SSS_{\veth_{i^j_{\kappa( j )}}}^{\nu_{i^j_{\kappa( j )}}} \circ \SSS_{\veth_{i^j_{\kappa( j ) - 1}}}^{\nu_{i^j_{\kappa( j ) - 1}}} \circ \dots \circ \SSS_{\veth_{i^j_1}}^{\nu_{i^j_1}},
\end{equation*}
each satisfying for $\mathbf y \in \mathbf X_j := \bigcap_{i \in V_j} lev_0( h_i )$ and every $\mathbf x$ with $d_{\mathbf X_j}( \mathbf x ) \geq \delta$:
\begin{equation}\label{eq:veiaString}
   \| \mathcal V_j( \mathbf x ) - \mathbf y \|^2 \leq \| \mathbf x - \mathbf y \|^2 - \epsilon^j_\delta.
\end{equation}
We can then average these operators to obtain a new feasibility operator as 
$\tilde{ \mathcal V}_{\mathbf X} := 1/Q\sum_{j = 1}^Q\mathcal V_j$. Making use of the triangle inequality and 
$(\sum_{i =  1}^n a_i)^2 \leq n\sum_{i =  1}^n a_i^2$, we have
\begin{eqnarray*}
   \| \tilde{\mathcal V}_{\X}( \mathbf x ) - \mathbf y \|^2 & {}= \left\| \frac1Q\sum_{j = 1}^Q\left[ \mathcal V_j( \mathbf x ) - \mathbf y \right]\right \|^2=\frac1{Q^2}\left\| \sum_{j = 1}^Q\left[ \mathcal V_j( \mathbf x ) - \mathbf y \right]\right \|^2 \\
   & {}\leq \frac1{Q^2}\left(\sum_{j = 1}^Q\left\| \mathcal V_j( \mathbf x ) - \mathbf y \right \| \right)^2\leq \frac1Q\sum_{j = 1}^Q\left\| \mathcal V_j( \mathbf x ) - \mathbf y \right \|^2.
\end{eqnarray*}
Now we notice that if $\mathbf X$ is bounded, $d_{\mathbf X}( \mathbf x ) \geq \delta$ implies that $\max\{ d_{\mathbf X_j}( \mathbf x ) \} \geq \tilde\delta > 0$ (for weaker conditions under which the same holds, see the results in~\cite{hof92}). Therefore, by using~(\ref{eq:veiaString}) in the inequality above we obtain, for $\mathbf y \in \mathbf X = \bigcap_{j = 1}^Q\mathbf X_j$ and every $\mathbf x$ with $d_{\mathbf X}( \mathbf x ) \geq \delta$:
\begin{eqnarray*}
   \| \tilde{\mathcal V}_{\X}( \mathbf x ) - \mathbf y \|^2 & {}\leq \frac1Q\sum_{j = 1}^Q\left\{\left\| \mathbf x - \mathbf y \right \|^2 - \epsilon^j_{\delta_j} \right\},
\end{eqnarray*}
where $\delta_j := d_{\mathbf X_j}( \mathbf x )$. Therefore:
\begin{eqnarray*}
   \| \tilde{\mathcal V}_{\X}( \mathbf x ) - \mathbf y \|^2 & {}\leq \left\| \mathbf x - \mathbf y \right \|^2 - \tilde\epsilon_{\delta},
\end{eqnarray*}
where $ \tilde\epsilon_\delta = \min_{j \in \{ 1, \dots, Q\}}\{\epsilon^j_{\tilde\delta}\}$.

The above argument suggests that if the operators $\mathcal V_j$ satisfy Condition~\ref{cond2} with $\mathbf X$ replaced by 
$\mathbf X_j$, then its average also will satisfy Condition~\ref{cond2} with $\mathbf X$ replaced by $\bigcap_{j=1}^{Q} \X_j$.
To the best of our knowledge, the previous discussion presents a first step to generalize some of the results 
from~\cite{censor01,censor2013convergence,censor2014string} towards averaging strings of inexact projections, 
or more specifically, averaging of Fej\'er-monotone operators.
We do not make use of averaged feasibility operators in this paper for clarity of presentation and also because our numerical 
examples can be handled in the classical way, without string averaging, since our model has few constraints. \qed
\end{remark}

With the optimality and feasibility operators already defined, we present a complete description of the algorithm we propose to solve 
the problem defined in (\ref{1-intro}), \textbf{(i)}-\textbf{(v)}.

\begin{algorithm}[String-averaging incremental subgradient method] \label{alg-proposed}        
$ $
\begin{description}
\item[\textbf{Input:}] Choose an initial vector $\vetx^0 \in \Y$ and a sequence of step-sizes $\lambda_k \geq 0$.
\item[\textbf{Iteration:}] Given the current iteration $\vetx^k$, do
\begin{description}
\item[Step 1.] \emph{(Step operators)} Compute independently for each $\ell=1,\dots,P$:
\begin{eqnarray} \label{step.operators.alg} \ds \vetx_{i_{0}^{\ell}}^{k} &=& \vetx^k, \nonumber \\
		 \ds \vetx_{i_{s}^{\ell}}^{k} &=& \F^{i_{s}^{\ell}}(\lambda_k, \vetx_{i_{s-1}^{\ell}}^{k}), \quad s = 1, \dots, m(l), \nonumber \\
		 \ds \vetx_{\ell}^{k} &=& \vetx_{i_{m(\ell)}^{\ell}}^{k},
\end{eqnarray}
where $i_{s}^{\ell} \in \Delta_\ell := S_\ell$ for each $s=1,\dots,m(\ell)$ and $\F^{i_{s}^{\ell}}$ is defined in \emph{(\ref{step.operators})}.

\item[Step 2.] \emph{(Combination operator)} Use the end-points $\vetx_{\ell}^{k}$ obtained in Step \emph{1} and the optimality 
operator $\OO_f$ defined in \emph{(\ref{SA_isoo})} to obtain:
\bb \label{opt.operator.alg} \vetx^{k+1/2} = \OO_f(\lambda_k, \vetx^k).\ee

\item[Step 3.] Apply feasibility operator $\V_{\X}$ defined in \emph{(\ref{viabilidade})} on the sub-iteration $\vetx^{k+1/2}$ 
to obtain:
\bb \label{feas.operator.alg} \vetx^{k+1} = \V_{\X}(\vetx^{k+1/2}).\ee

\item[Step 4.] Update $k$ and return to Step \emph{1}.
\end{description}
\end{description}
\end{algorithm}

\section{Convergence analysis} \label{sec.4}

Along this section, we denote $ F_{S_\ell}(\vetx) = \sum_{s=1}^{m(\ell)} f_{i_{s}^{\ell}}(\vetx)$ for each $\ell = 1, \dots, P$.
The following subgradient boundedness assumption is key in this paper: for all $\ell$ and $s$,  
\bb \label{limitacao} \ds C_{i_{s}^{\ell}} = \sup_{k \geq 0}\left\{\left\|\vetg \right\| \, | \, \vetg \in \partial f_{i_{s}^{\ell}}(\vetx^k) \cup \partial f_{i_{s}^{\ell}}(\vetx_{i_{s-1}^{\ell}}^{k}) \right\} < \infty.\ee 
Recall that Theorem \ref{conv1} is the main tool for the convergence analysis, so we will show that each of its conditions are 
valid under assumption (\ref{limitacao}). We present auxiliary results in the next two lemmas.

\begin{lemma} \label{lem1} Let $\{\vetx^k \}$ be the sequence generated by Algorithm {\em \ref{alg-proposed}} 
and suppose that subgradient boundedness assumption \emph{(\ref{limitacao})} holds. Then, for each $\ell$ and $s$ and for all $k \geq 0$, we have 
\begin{description}
\item[(i)] \bb \label{ineq1} \ds f_{i_{s}^{\ell}}(\vetx^k) - f_{i_{s}^{\ell}}(\vetx_{i_{s-1}^{\ell}}^{k}) \leq C_{i_{s}^{\ell}} \|\vetx_{i_{s-1}^{\ell}}^{k} - \vetx^k \|. \ee
\item[(ii)] \bb \label{ineq2} \ds \|\vetx_{i_{s}^{\ell}}^{k} - \vetx^k \| \leq \lambda_k \sum_{r=1}^{s} C_{i_{r}^{\ell}}.\ee
\item[(iii)] For all $\vety \in \R^n$, we have
\bb \label{ineq3} \ds \left\langle \sum_{s=1}^{m(\ell)} \vetg_{i_{s}^{\ell}}^{k}, \vety - \vetx^k \right\rangle \leq F_{S_{\ell}}(\vety) - F_{S_{\ell}}(\vetx^k) + 2\lambda_k \sum_{s=2}^{m(\ell)} C_{i_{s}^{\ell}}\left(\sum_{r=1}^{s-1} C_{i_{r}^{\ell}}\right), \ee
where $\vetg_{i_{s}^{\ell}}^{k} \in \partial f_{i_{s}^{\ell}}(\vetx_{i_{s-1}^{\ell}}^{k})$.
\end{description}
\end{lemma}
\begin{proof}
\begin{description}
\item[(i)] By definition of the subdifferential $\ds \partial f_{i_{s}^{\ell}}(\vetx^{k})$, we have
\begin{equation*} \ds f_{i_{s}^{\ell}}(\vetx^k) - f_{i_{s}^{\ell}}(\vetx_{i_{s-1}^{\ell}}^{k}) \leq - \langle \vetv_{i_{s}^{\ell}}^{k}, \vetx_{i_{s-1}^{\ell}}^{k} - \vetx^k \rangle,\end{equation*}
where $\vetv_{i_{s}^{\ell}}^{k} \in \partial f_{i_{s}^{\ell}}(\vetx^k)$. 
The result follows from the Cauchy-Schwarz inequality and the subgradient boundedness assumption (\ref{limitacao}).

\item[(ii)] Developing the equation $\ds \vetx_{i_{s}^{\ell}}^{k} = \vetx_{i_{s-1}^{\ell}}^{k} - \lambda_k \vetg_{i_{s}^{\ell}}^{k}$ for each $s=1, \dots,m(\ell)$ yields,
\begin{equation*}\begin{array}{lllll} \ds \|\vetx_{i_{1}^{\ell}}^{k} - \vetx^k \| &=& \ds \| \vetx^k - \lambda_k \vetg_{i_{1}^{\ell}}^{k} - \vetx^k \| &\leq& \lambda_k C_{i_{1}^{\ell}},
	\\  \ds \| \vetx_{i_{2}^{\ell}}^{k} - \vetx^k \| &\leq& \ds \| \vetx_{i_{1}^{\ell}}^{k} - \vetx^k \| + \lambda_k \| \vetg_{i_{2}^{\ell}}^{k} \| &\leq& \lambda_k (C_{i_{1}^{\ell}} + C_{i_{2}^{\ell}}),
	\\ \qquad \vdots & & \qquad \qquad \vdots & & \qquad \vdots
	\\  \ds \| \vetx_{i_{s}^{\ell}}^{k} - \vetx^k \| &\leq& \| \vetx_{i_{s-1}^{\ell}}^{k} - \vetx^k \| + \lambda_k \| \vetg_{i_{s}^{\ell}}^{k} \| &\leq& \ds \lambda_k \sum_{r=1}^{s} C_{i_{r}^{\ell}}. \end{array} \end{equation*}

\item[(iii)] By Cauchy-Schwarz inequality and definition of the subdifferential $\ds \partial f_{i_{s}^{\ell}}(\vetx_{i_{s-1}^{\ell}}^{k})$ we have,
$$\begin{array}{lll} \ds \left\langle \sum_{s=1}^{m(\ell)} \vetg_{i_{s}^{\ell}}^{k}, \vety - \vetx^k \right\rangle &=& \ds \sum_{s=1}^{m(\ell)} \langle  \vetg_{i_{s}^{\ell}}^{k}, \vetx_{i_{s-1}^{\ell}}^{k} - \vetx^k \rangle + \sum_{s=1}^{m(\ell)} \langle \vetg_{i_{s}^{\ell}}^{k}, \vety - \vetx_{i_{s-1}^{\ell}}^{k} \rangle \\
														   &\leq& \ds \sum_{s=1}^{m(\ell)} \| \vetg_{i_{s}^{\ell}}^{k} \| \|\vetx^k - \vetx_{i_{s-1}^{\ell}}^{k} \| + \sum_{s=1}^{m(\ell)} ( f_{i_{s}^{\ell}}(\vety) - f_{i_{s}^{\ell}}(\vetx_{i_{s-1}^{\ell}}^{k}) ) \\
														   &=& \ds \sum_{s=2}^{m(\ell)} \| \vetg_{i_{s}^{\ell}}^{k} \| \|\vetx^k - \vetx_{i_{s-1}^{\ell}}^{k} \| + F_{S_{\ell}}(\vety) - F_{S_{\ell}}(\vetx^k) \\
														   && \qquad \ds - \sum_{s=2}^{m(\ell)} (f_{i_{s}^{\ell}}(\vetx_{i_{s-1}^{\ell}}^{k}) - f_{i_{s}^{\ell}}(\vetx^k) ).
\end{array}$$
By eqs. (\ref{ineq1}), (\ref{ineq2}) and the subgradient boundedness assumption (\ref{limitacao}) we obtain,
$$\begin{array}{lll}
\ds \left\langle \sum_{s=1}^{m(\ell)} \vetg_{i_{s}^{\ell}}^{k}, \vety - \vetx^k \right\rangle &\leq& \ds \sum_{s=2}^{m(\ell)} \| \vetg_{i_{s}^{\ell}}^{k} \| \| \vetx^k - \vetx_{i_{s-1}^{\ell}}^{k} \| + F_{S_{\ell}}(\vety) - F_{S_{\ell}}(\vetx^k) \\
												 && \qquad \ds + \sum_{s=2}^{m(\ell)}  \|\vetv_{i_{s}^{\ell}}^{k} \| \| \vetx^k - \vetx_{i_{s-1}^{\ell}}^{k} \| \\
												 &\leq& \ds F_{S_{\ell}}(\vety) - F_{S_{\ell}}(\vetx^k) + \sum_{s=2}^{m(\ell)} ( \| \vetg_{i_{s}^{\ell}}^{k} \| + \| \vetv_{i_{s}^{\ell}}^{k} \| )\left(\lambda_k \sum_{r=1}^{s-1} C_{i_{r}^{\ell}}\right) \\
												 &\leq& \ds F_{S_{\ell}}(\vety) - F_{S_{\ell}}(\vetx^k) + 2 \lambda_k \sum_{s=2}^{m(\ell)} C_{i_{s}^{\ell}}\left(\sum_{r=1}^{s-1} C_{i_{r}^{\ell}}\right).
\end{array}$$
\end{description}
\end{proof}

The following Lemma is useful to analyze the convergence of the Algorithm \ref{alg-proposed}.

\begin{lemma} \label{lem2} Let $\{\vetx^k \}$ be the sequence generated by Algorithm {\em \ref{alg-proposed}} 
and suppose that assumption \emph{(\ref{limitacao})} holds.
Then, there is a positive constant $C$ such that, for all $\vety \in \Y \supset \X$ and for all $k \geq 0$ we have
\bb \label{desig} \ds \| \OO_f(\lambda_k, \vetx^k) - \vety \|^2 \leq \| \vetx^{k} - \vety \|^2 - \frac{2}{P} \lambda_{k} (f(\vetx^{k}) - f(\vety)) + C \lambda_{k}^2, \ee
\end{lemma}

\begin{proof}
Initially, we can develop equation (\ref{step.operators.alg}) for each $\ell = 1, \dots, P$ and obtain $\ds \vetx_{\ell}^{k} = \vetx^k - \lambda_k \sum_{s=1}^{m(\ell)} \vetg_{i_{s}^{\ell}}^{k}$, 
where $\vetg_{i_{s}^{\ell}}^{k} \in \partial f_{i_{s}^{\ell}}(\vetx_{i_{s-1}^{\ell}}^{k})$. Thus, from equation 
(\ref{opt.operator.alg}) we have for all $k \geq 0$,
\begin{equation*} \begin{array}{lll} 
\ds \OO_f(\lambda_k, \vetx^k) &=& \ds \sum_{\ell=1}^{P} w_{\ell} \vetx_{\ell}^{k} \\
		&=& \ds \sum_{\ell=1}^{P} w_{\ell} \left(\vetx^k - \lambda_k \sum_{s=1}^{m(\ell)} \vetg_{i_{s}^{\ell}}^{k} \right) \\
		&=& \ds \vetx^k - \lambda_k \sum_{\ell=1}^{P} w_{\ell} \sum_{s=1}^{m(\ell)} \vetg_{i_{s}^{\ell}}^{k}.
\end{array}\end{equation*}	
Using the above equation we obtain for all $\vety \in \Y$ and for all $k \geq 0$,
$$\begin{array}{lll}
\ds \| \OO_f(\lambda_k, \vetx^k) - \vety \|^2 &=& \ds \left\| \vetx^k - \vety - \lambda_k \sum_{\ell=1}^{P} w_\ell \sum_{s=1}^{m(\ell)} \vetg_{i_{s}^{\ell}}^{k} \right\|^2 \\
					 &=& \ds \|\vetx^k - \vety \|^2 - 2 \left\langle \vetx^k - \vety, \, \lambda_k \sum_{\ell=1}^{P} w_\ell \sum_{s=1}^{m(\ell)} \vetg_{i_{s}^{\ell}}^{k} \right\rangle + \left\| \lambda_k \sum_{\ell=1}^{P} w_\ell \sum_{s=1}^{m(\ell)} \vetg_{i_{s}^{\ell}}^{k} \right\|^2 \\
					 &=& \ds \|\vetx^k - \vety \|^2 + 2 \lambda_k \sum_{\ell=1}^{P} w_\ell \left\langle \sum_{s=1}^{m(\ell)} \vetg_{i_{s}^{\ell}}^{k}, \, \vety - \vetx^k \right\rangle + \lambda_{k}^{2} \left\| \sum_{\ell=1}^{P} w_\ell \sum_{s=1}^{m(\ell)} \vetg_{i_{s}^{\ell}}^{k} \right\|^2.
\end{array}$$
Now, using Lemma \ref{lem1} \textbf{(iii)}, triangle inequality and $P\sum_{\ell=1}^{P} w_\ell F_{S_{\ell}}(\vetx) = f(\vetx)$ we have,
$$\begin{array}{lll}
\ds \|\OO_f(\lambda_k, \vetx^k) - \vety \|^2 &\leq& \ds \| \vetx^k - \vety \|^2 - 2\lambda_k \sum_{\ell=1}^{P} w_\ell \left[F_{S_{\ell}}(\vetx^k) - F_{S_{\ell}}(\vety) - 2\lambda_k\sum_{s=2}^{m(\ell)} C_{i_{s}^{\ell}} \left(\sum_{r=1}^{s-1} C_{i_{r}^{\ell}}\right)\right] \\
					 && \ds \qquad + \lambda_{k}^{2} \left\| \sum_{\ell=1}^{P} w_\ell \sum_{s=1}^{m(\ell)} \vetg_{i_{s}^{\ell}}^{k}\right\|^2 \\
					 &\leq& \ds \| \vetx^k - \vety \|^2 - 2 \lambda_k \left(\sum_{\ell=1}^{P} w_\ell F_{S_{\ell}}(\vetx^k) - \sum_{\ell=1}^{P} w_\ell F_{S_{\ell}}(\vety) \right) \\
					 && \ds \qquad + 4 \lambda_{k}^{2} \sum_{\ell=1}^{P} w_\ell \left[\sum_{s=2}^{m(\ell)} C_{i_{s}^{\ell}} \left(\sum_{r=1}^{s-1} C_{i_{r}^{\ell}}\right)\right] + \lambda_{k}^{2} \left( \sum_{\ell=1}^{P} w_\ell \sum_{s=1}^{m(\ell)} \| \vetg_{i_{s}^{\ell}}^{k} \| \right)^2 \\
					 &=& \ds \| \vetx^k - \vety \|^2 - 2\frac{\lambda_k}{P} (f(\vetx^k) - f(\vety)) + 4\lambda_{k}^{2}\sum_{\ell=1}^{P} w_\ell \left[\sum_{s=2}^{m(\ell)} C_{i_{s}^{\ell}}\left(\sum_{r=1}^{s-1} C_{i_{r}^{\ell}}\right)\right] \\
					 && \ds \qquad + \lambda_{k}^{2} \left( \sum_{\ell=1}^{P} w_\ell \sum_{s=1}^{m(\ell)} \| \vetg_{i_{s}^{\ell}}^{k} \| \right)^2.
\end{array}$$
Finally, by subgradient boundedness assumption (\ref{limitacao}), we obtain for all $\vety \in \Y$ and for all $k \geq 0$,
$$\begin{array}{lll}
\ds \| \OO_f(\lambda_k, \vetx^k) - \vety \|^2 &\leq& \ds \| \vetx^k - \vety \|^2 - 2\frac{\lambda_k}{P} (f(\vetx^k)- f(\vety)) \\
					 && \qquad + \ds \lambda_{k}^{2} \left[ 4\sum_{\ell=1}^{P} w_\ell \left[\sum_{s=2}^{m(\ell)} C_{i_{s}^{\ell}}\left(\sum_{r=1}^{s-1} C_{i_{r}^{\ell}} \right) \right] + \left( \sum_{\ell=1}^{P} w_\ell \sum_{s=1}^{m(\ell)} C_{i_{s}^{\ell}} \right)^2 \right] \\
					 &=& \ds \| \vetx^{k} - \vety \|^2 - \frac{2}{P} \lambda_{k} (f(\vetx^{k}) - f(\vety)) + C \lambda_{k}^2.
\end{array}$$
\end{proof}

The next two propositions aim at showing
that, under some mild additional hypothesis, $\OO_f$ and $\V_{\X}$ satisfy Conditions~\ref{cond1}-\ref{cond2}.

\begin{proposition} \label{prop otim} Let $\{\vetx^k \}$ be the sequence generated by Algorithm {\em \ref{alg-proposed}} and 
suppose that subgradient boundedness assumption \emph{(\ref{limitacao})} holds. 
Then, if $\lambda_k \rightarrow 0^{+}$, the optimality operator $\OO_f$ satisfies Condition \emph{\ref{cond1}} of Theorem 
\emph{\ref{conv1}}.
\end{proposition}
\begin{proof}
Lemma \ref{lem2} ensures that for all $\vetx \in \X \subset \Y$ we have,
\begin{equation*} \ds \| \OO_f (\lambda_k, \vetx^k) - \vetx \|^2 \leq \| \vetx^k - \vetx \|^2 - \frac{2}{P} \lambda_k (f(\vetx^k) - f(\vetx)) + C \lambda_{k}^{2}. \end{equation*}
Defining $\alpha = 2 / P$ and $\rho_k = \lambda_k C \geq 0$, equation~(\ref{cond1-eq1}) is satisfied and $\rho_k \to 0$. 
Furthermore, by triangle inequality and subgradient boundedness assumption (\ref{limitacao}) we have,
\begin{eqnarray*} \ds \| \OO_f (\lambda_k, \vetx^k) - \vetx^k \| &=& \ds \left\|\sum_{\ell=1}^{P} w_{\ell} \vetx_{\ell}^{k} - \vetx^k\right\| \\
								      &=& \ds \left\| \vetx^k - \lambda_k \sum_{\ell=1}^{P} w_\ell \sum_{s=1}^{m(\ell)} \vetg_{i_{s}^{\ell}}^{k} - \vetx^k \right\| \\
								      &=& \ds \lambda_k \left\|\sum_{\ell=1}^{P} w_\ell \sum_{s=1}^{m(\ell)} \vetg_{i_{s}^{\ell}}^{k} \right\| \\
								      &\leq& \ds \lambda_k \sum_{\ell=1}^{P} w_\ell \sum_{s=1}^{m(\ell)} C_{i_{s}^{\ell}},
\end{eqnarray*}
implying that equation~(\ref{cond1-eq2}) is satisfied with $\gamma = \sum_{\ell=1}^{P} w_\ell \sum_{s=1}^{m(\ell)} C_{i_{s}^{\ell}}$. Therefore, 
Condition~\ref{cond1} is satisfied.
\end{proof}

\begin{proposition} \emph{(\cite{helou09}, Proposition 3.4)} \label{prop viab} 
Let $\vetx^{k+1/2}$ given in \emph{(\ref{opt.operator.alg})} be the first element $\vets_{0}^{k}$ of the sequence 
$\{ \vets_{i}^{k} \}$, $i = 1, \dots, t$, given as $\vets_{i}^{k}:= \SSS_{h_i}^{\nu_i}(\vets_{i-1}^{k})$. In this sense, consider 
that $\veth_{i}^{k} \in \partial h_i(\vets_{i-1}^{k})$. Suppose that for some index $j$, the set $lev_0(h_j)$ is bounded. 
In addition, consider that all sequences 
$\left\{\veth_{i}^{k}\right\}$ are bounded.
Then, $\V_{\X}$ satisfies Condition \emph{\ref{cond2}} of Theorem \emph{\ref{conv1}}.
\end{proposition}

The main result of the paper is given next.
\begin{corollary} \label{convSA} Let $\{ \vetx^k \}$ be the sequence generated by Algorithm \emph{\ref{alg-proposed}} and 
suppose that subgradient boundedness assumption \emph{(\ref{limitacao})} holds. In addition, suppose that $lev_0(h_j)$ is bounded 
for some $j$ and all sequences $\{ \veth_{i}^{k} \}$ are bounded. 
Then, if Conditions \emph{\ref{cond3}}-\emph{\ref{cond4}} of Theorem \emph{\ref{conv1}} hold, we have
\begin{equation*} \ds d_{\X^{\ast}}(\vetx^k) \rightarrow 0 \qquad \mbox{and} \qquad \lim_{k \rightarrow \infty} f(\vetx^k) = f^{\ast}.\end{equation*}
\end{corollary}
\begin{proof}
Propositions \ref{prop otim} and \ref{prop viab} state that operators $\OO_f$ and $\V_{\X}$ 
satisfy Conditions \ref{cond1}-\ref{cond2}. Therefore, the result is obtained applying Theorem \ref{conv1}.
\end{proof}

Recall that we discuss the reasonability of the Condition \ref{cond4} as a hypothesis for this corollary in Remark \ref{rmk1}.

\section{Numerical experiments} \label{sec.5}

In this section, we apply the problem formulation (\ref{1-intro}), \textbf{(i)}-\textbf{(v)} and the method given in Algorithm 
\ref{alg-proposed} to the reconstruction of tomographic images from few views, and we explore results obtained from simulated and 
real data to show that 
the method is competitive when compared with the classic incremental subgradient algorithm. Let us start with a brief description 
of the problem.
The task of reconstructing tomographic images is related to the mathematical problem of finding a function $\psi:\mathbb{R}^2 \rightarrow \mathbb{R}$ from its line integrals along straight lines. More specifically, we desire to find $\psi$ given the following function:
\bb \label{radon_transform} \ds \mathcal{R}[\psi](\theta, t) := \int_{\mathbb{R}} \psi(t(\cos \theta, \sin \theta)^T + s(- \sin \theta, \cos \theta)^T) \, ds. \ee
The application $\psi \mapsto \mathcal{R}[\psi]$ is so-called \textit{Radon Transform} and for a fixed $\theta$, $\mathcal{R}_{\theta}[\psi](t)$ is known as a \textit{projection} of $\psi$.
For a detailed discussion about the physical and mathematical aspects involving tomography and the definition in 
(\ref{radon_transform}), see, for example \cite{natterer86, natterer2001mathematical, herman2009fundamentals}.

We now provide an example to better understand the geometric meaning of the definition of Radon transform. 
We can display $\psi$ as a picture if we assign to each value in $[0,1]$, a grayscale such as in Figure 
\ref{example_radon}-(a). Here we use an artificial image made up of a sum of 
indicator functions of ellipses. The bar on the right indicates the grayscale used. 
We also show the axes $t$, $x$, $y$ and the integration path for a given pair $(\theta, t')$, which appears 
as the dashed line segment. The $t$-axis directions vary according to the number of angles adopted for the 
reconstruction process. 
In general, $\theta \in [0, \pi)$ because $\mathcal{R}[\psi](\theta + \pi, -t) = \mathcal{R}[\psi](\theta,t)$. 
For a fixed angle $\theta$, the projection $\mathcal{R}_{\theta}[\psi](t)$ is computed for each $t' \in [-1,1]$. 
Figure \ref{example_radon}-(b) shows the projections obtained for three fixed angles $\theta$: $\theta_1$, $\theta_2$ and $\theta_3$. 
Its representation given in Figure \ref{example_radon}-(c) as an image in the $\theta \times t$ coordinate system is called 
\textit{sinogram}. We also call the Radon transform at a fixed angle a \emph{view} or \emph{projection}.
\begin{figure}
\centering
\includegraphics[width=\textwidth]{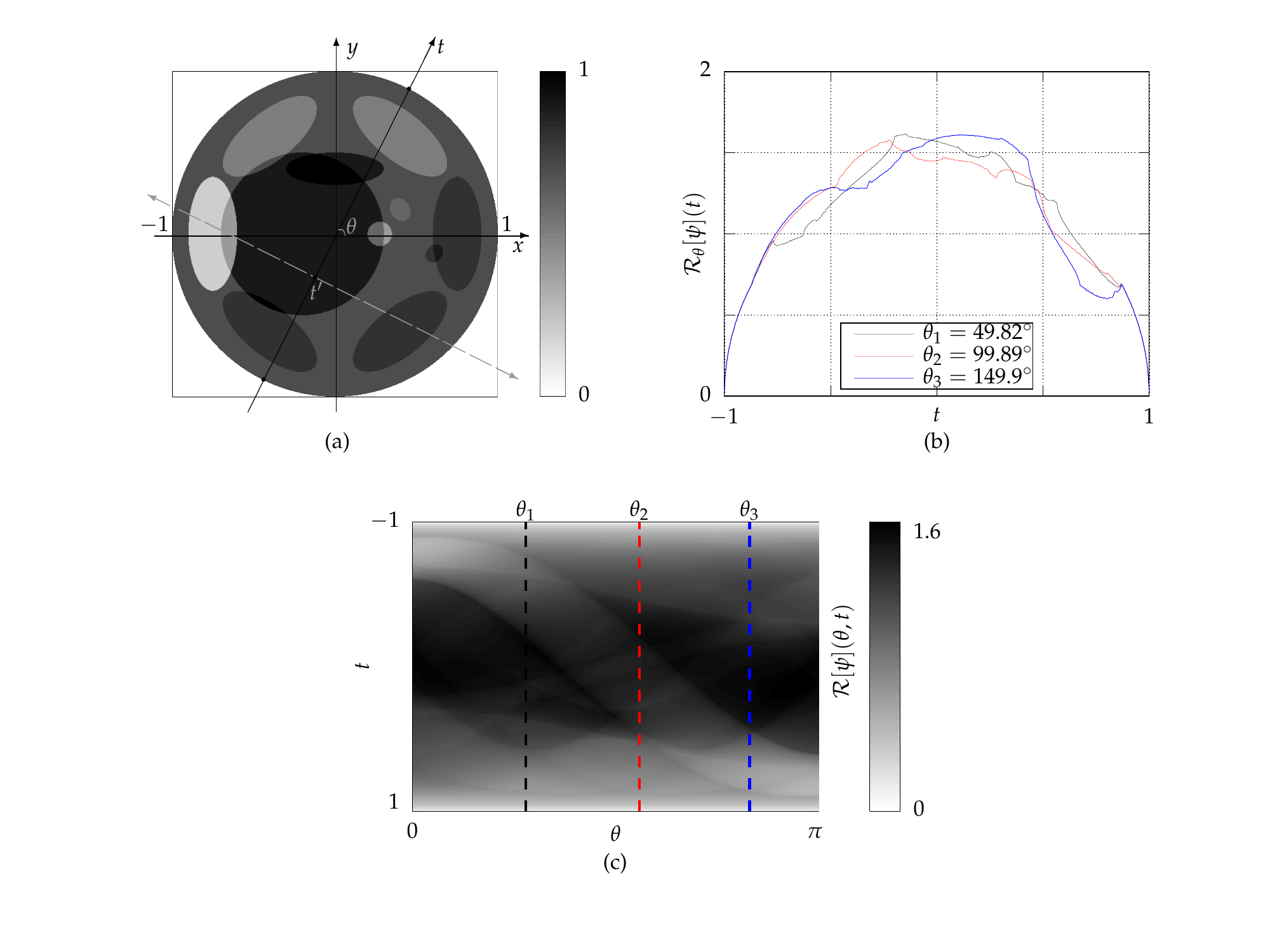}
\caption{(a) An artificial image and the integration path for a given $(\theta, t')$ used to compute a projection 
$\mathcal{R}_{\theta}[\psi](t)$. (b) Projections for three fixed angles. (c) Sinogram obtained from image in (a).} \label{example_radon}
\end{figure}

The Radon transform models the data in a tomographic image reconstruction problem. 
That is, for reconstruction of the function $\psi$, we must go from Figure~\ref{example_radon}-(c) to the desired 
image in Figure~\ref{example_radon}-(a), i.e., it would be desirable to calculate the inverse 
$\mathcal{R}^{-1}$. However, as already mentioned, the Radon Transform is a compact operator and therefore its inversion is 
an ill-conditioned problem. In fact, for $n=2$ and $n=3$, Radon obtained inversion formulas involving first and 
second order differentiation of the data \cite{natterer2001mathematical}, respectively, implying in an unstable process 
due the increase of the error propagation in presence of perturbed data (when noise is present, which may occur due to width 
of the x-ray beam, scatter, hardening of the beam, photon statistics, detector inaccuracies, etc \cite{herman2009fundamentals}). 
Other difficulties arise when using analytical
solutions in practice due to, for example, the limited number of views that often occurs.
This is why more sophisticated optimization models are useful, and it is desirable to use an objective function and constraints that forces the consistency of the solution to the data and guarantees stability of the solution.

\subsection{Experimental work}
In what follows we provide a detailed description of the experimental setup.

\begin{description}
\item a) \textit{The problem}: we consider the task of reconstructing an image from few views. We use a model based in the 
$\ell_1$-norm of the residual associated to the linear system $R \vetx = \imb$, where $R$ is the $m \times n$ Radon matrix, obtained 
through discretization of the Radon transform in (\ref{radon_transform}), 
$\vetx \in \R^n$ is the solution that we want to find, $\imb \approx R \vetx^{\ast} \in \mathbb{R}^m$ represents the data that we have 
for the reconstruction (sinogram), $\vetx^{\ast} \in \mathbb{R}_{+}^{n}$ is the original
image and $m \ll n$. The choice of the $\ell_1$-norm serves as a way to promote robustness to the error $\imb - R \vetx^{\ast}$, which in the case of synchrotron illuminated 
tomography has relatively few very large components and many smaller ones. The small errors are related with the Poisson nature of the data, while the outliers happen because of systematic 
detection failure either due to dust in the ray path or to, e.g., failed detector pixels. 
In this manner, the following optimization problem has suitable features for the use of the Algorithm \ref{alg-proposed}:
\bb \label{problem_rn1} \begin{array}{l} \min \, \, f(\vetx) = \left\| R \vetx - \imb \right\|_1
\\
\mbox{s.t.} \quad h(\vetx) = TV(\vetx) - \tau \leq 0,
\\
\qquad \, \vetx \in \mathbb{R}_{+}^{n}. \end{array} \ee
Note that the objective function $f(\vetx) = \sum_{i=1}^{m} | \langle \vetr_i, \vetx \rangle - b_i| = \sum_{i=1}^{m} f_i(\vetx)$, 
where $\vetr_i$ represents the $i$-th row of $R$. In comparison to (\ref{1-intro}) \textbf{(i)}-\textbf{(v)}, model 
(\ref{problem_rn1}) suggests constant 
weights $w_\ell = 1 / P$ for all $\ell = 1, \dots, P$ to satisfy conditions \textbf{(iv)} and \textbf{(v)}. 
In our tests, we use $P = 1, \dots, 6$ and, to build the sets $S_\ell$, we ordered 
the indices of the data randomly and then distributed in $P$ equally sized sets (or as close to it as possible) 
aiming at satisfying condition \textbf{(iii)}. We also assume that the image $\vetx^{\ast}$ to be reconstructed has large 
approximately constant areas, 
as is often the case in tomographic images.
Operator $TV: \mathbb{R}^n \rightarrow \mathbb{R}_{+}$ is called \textit{total variation} and is defined by
\begin{equation*} TV(\vetx) = \sum_{i=1}^{r_2} \sum_{j=1}^{r_1} \sqrt{\left(x_{i,j} - x_{i-1,j}\right)^2 + \left( x_{i,j} - x_{i,j-1}\right)^2},\end{equation*}
where $\vetx = [x_q]^T$, $q \in \{1, \dots, n \}$, $n = r_1 r_2$ and $x_{i,j} := x_{(i-1)r_1 + j}$.
We have also used the boundary conditions $x_{0,j} = x_{i,0} = 0$ and $\tau = TV(\vetx^{\ast})$.

\item b) \textit{Data generation}: for this simulated experiment, we have considered the reconstruction of the Shepp-Logan 
phantom \cite{kas88}. 
In Figure \ref{phantom}, we show this image using a grayscale version with resolution of $256 \times 256$. 
This resolution will also be used for the reconstruction.
\begin{figure}[htbp!]
\centering
\includegraphics[scale=0.5]{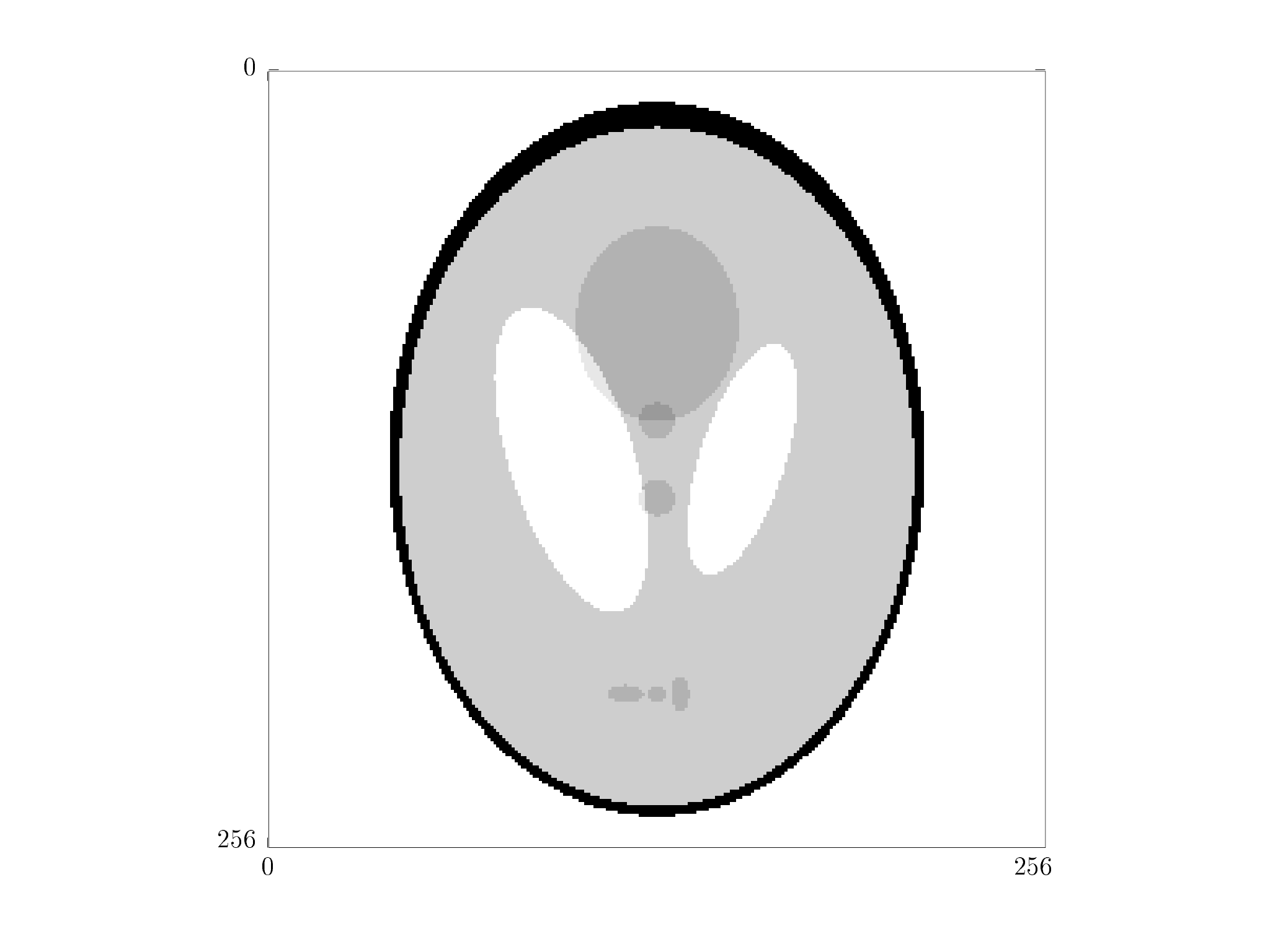}
\caption{Shepp-Logan phantom with resolution $256 \times 256$.}\label{phantom}
\end{figure}
For the vector $\imb$ that contains the data to be used in the reconstruction, we need an efficient routine for calculating the 
product $R \vetx $. We consider $24$ equally spaced angular projections with each sampled at $256$ equally spaced 
points.
We also consider reconstruction of images affected by Poisson noise, i.e., we execute the algorithms using data that was 
generated as samples of a Poisson random variable having as parameter the exact Radon Transform 
of the scaled phantom:
\bb \imb \sim Poisson \left( \kappa \mathcal{R}[\psi](\theta, t) \right), \ee
where the scale factor $\kappa > 0$ is used to control the simulated photon count, i.e., the noise level. Figure \ref{data} shows 
the result obtained for $\imb = R \vetx^{\ast}$, where $\vetx^{\ast}$ is the Shepp-Logan phantom, in both cases, i.e., with and 
without noise in the data.
\begin{figure}[htbp!]
\centering
\subfloat[Noise-free.]{
\includegraphics[scale=0.35]{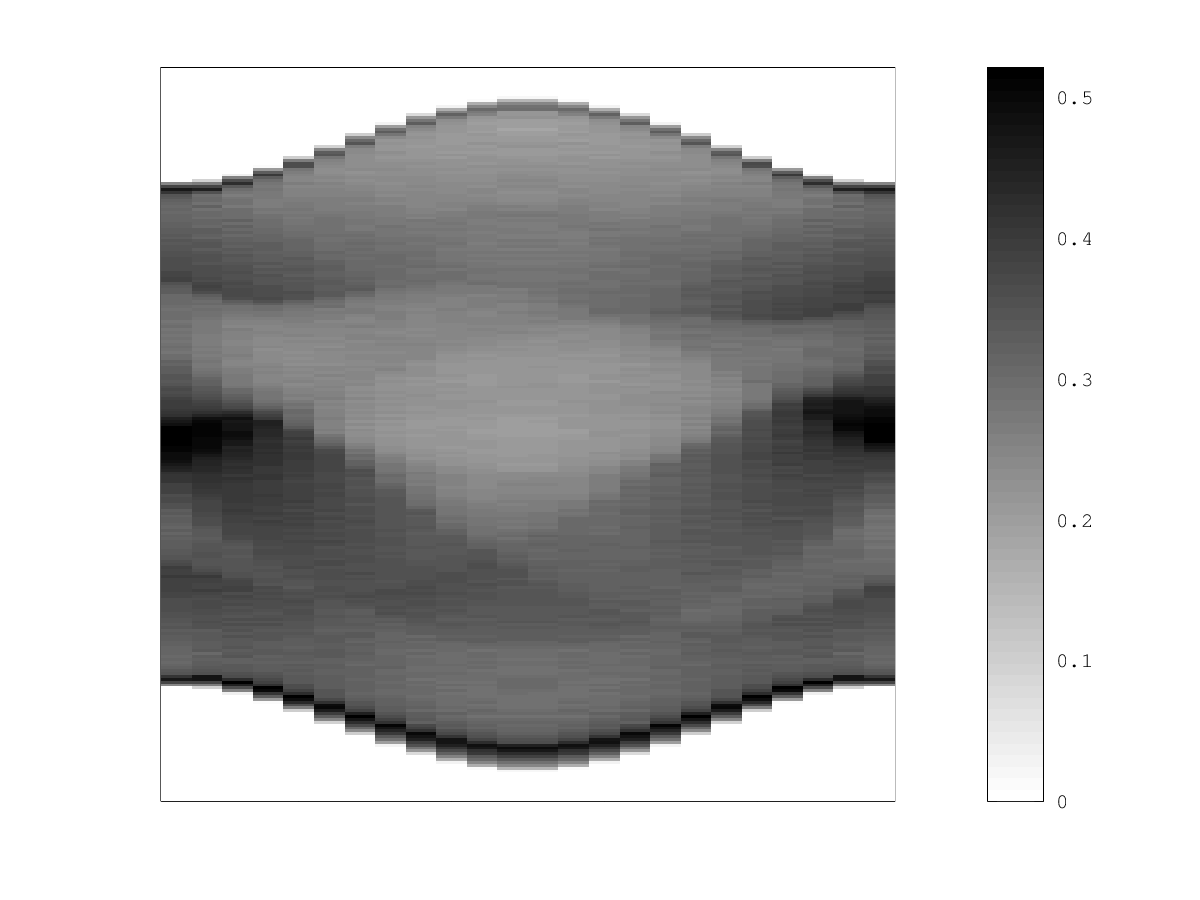}
}
\subfloat[$\kappa = 10^2$.]{
\includegraphics[scale=0.35]{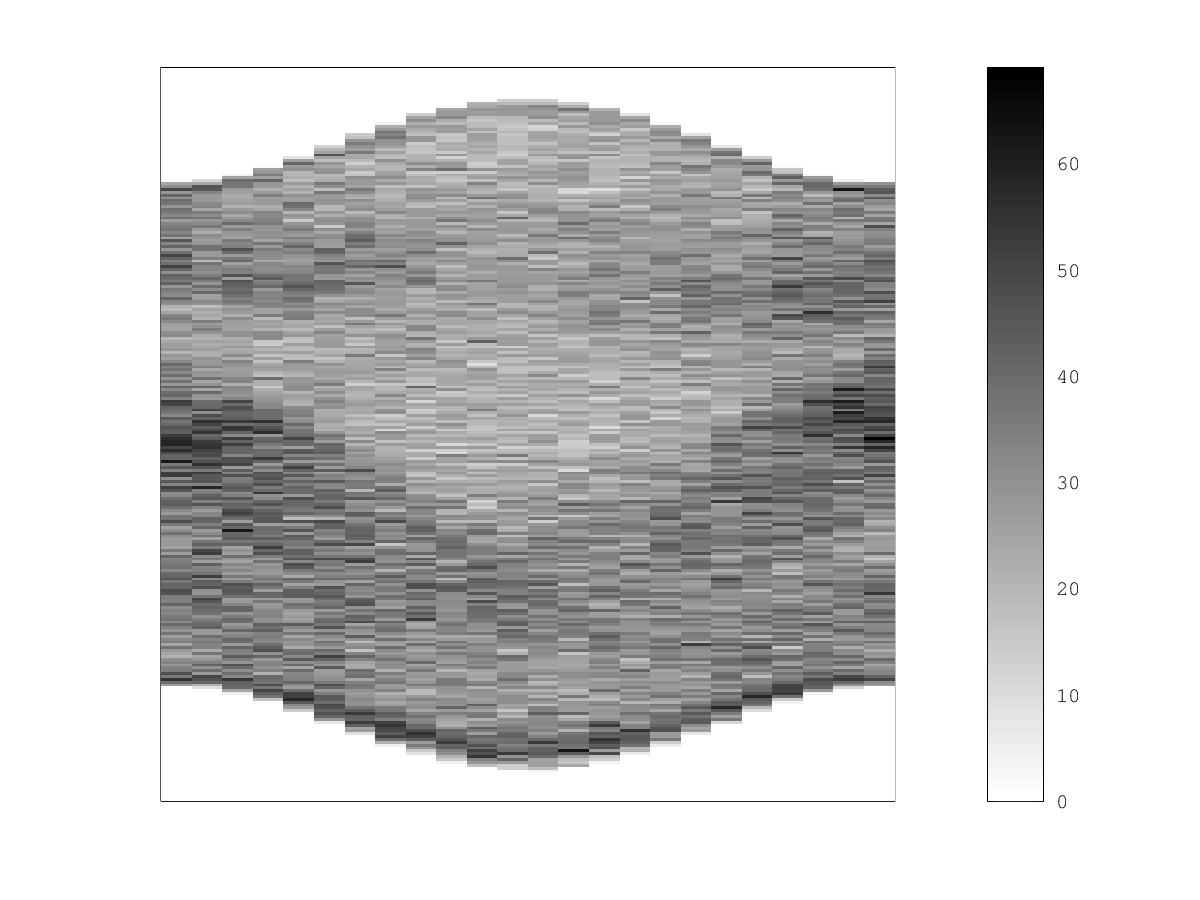}
}
\\ 
\subfloat[$\kappa = 4 \times 10^2$.]{
\includegraphics[scale=0.35]{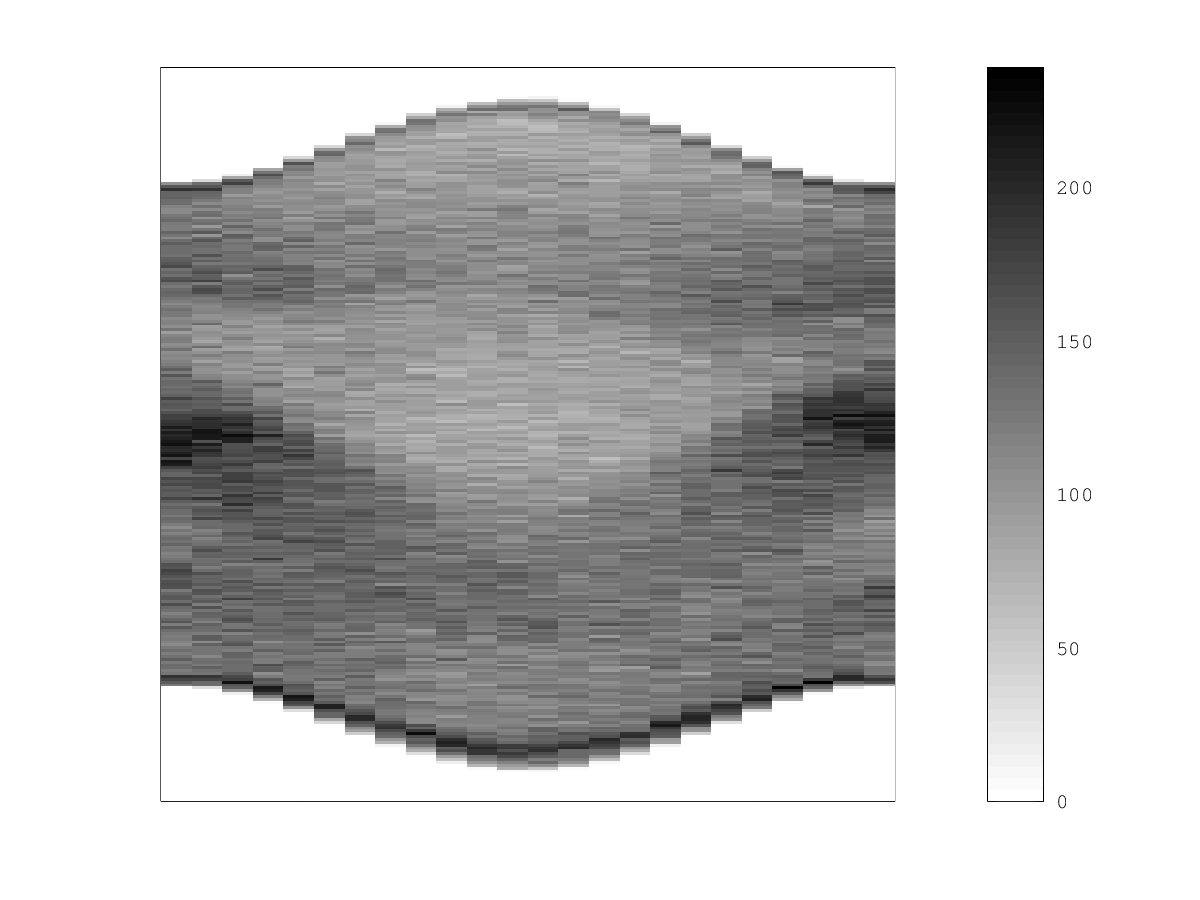}
}
\subfloat[$\kappa = 10^3$.]{
\includegraphics[scale=0.35]{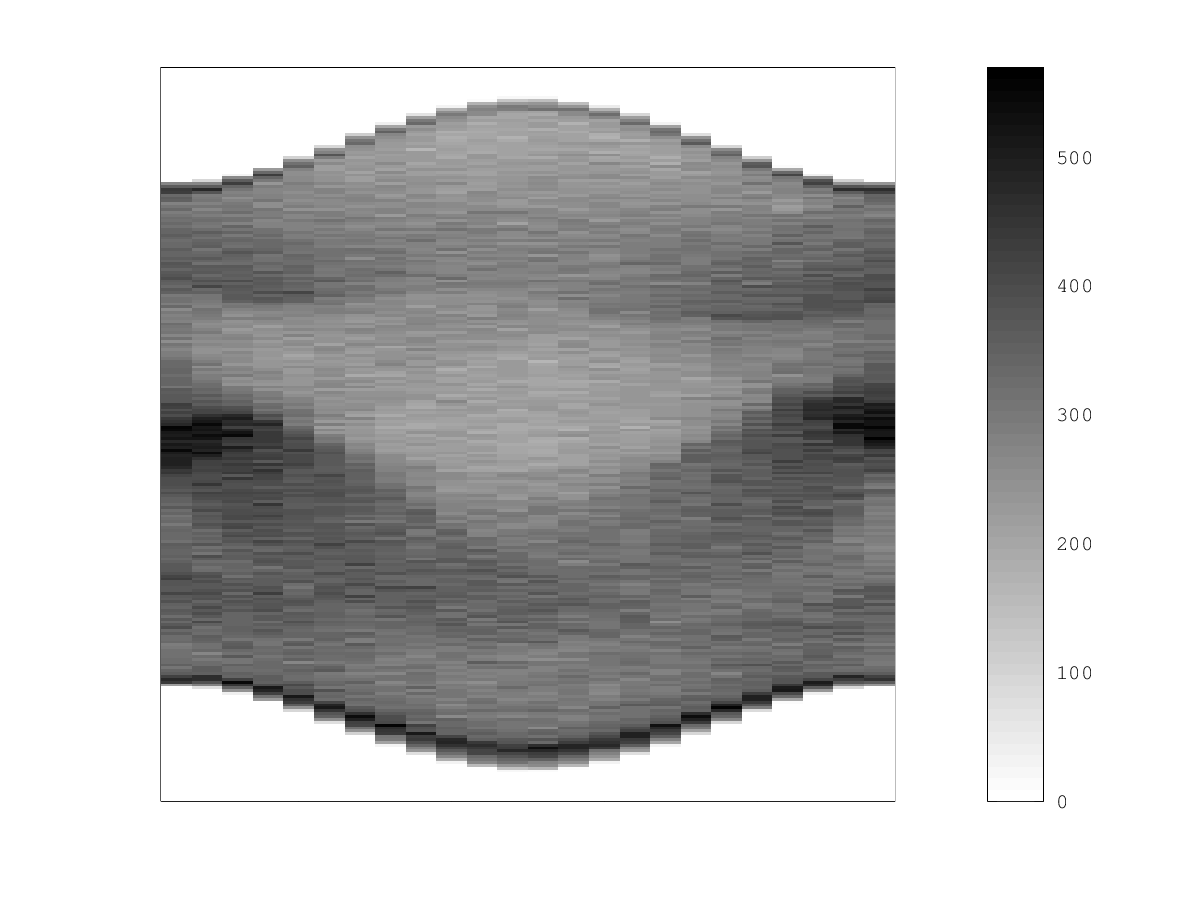}
}
\caption{Sinograms obtained from Radon transform of Shepp-Logan phantom. Only $24$ equally spaced 
angular projections are taken, each sampled at $256$ equally spaced points.} \label{data}
\end{figure}

\item c) \textit{Initial image}: for the initial image $\vetx^0$, we seek an uniform vector that somehow has information 
from data obtained by the Radon Transform of Shepp-Logan head phantom. For that, we use an initial image that satisfy
$ \sum_{i=1}^{m} \langle \vetr_i, \vetx^0 \rangle = \sum_{i=1}^{m} b_i$.
Therefore, supposing $x_{j}^{0} = \zeta$ for all $j = 1, \dots, n$, we can compute $\zeta$ by
\bb \label{initial_image} \ds \zeta = \frac{\sum_{i=1}^{m} b_i}{\sum_{i=1}^{m} \langle \vetr_i, \mathbf{1} \rangle},\ee
where $\mathbf{1}$ is the $n$-vector whose components are all equal to $1$.

\item d) \textit{Applying Algorithm \emph{\ref{alg-proposed}}}: 
step-size sequence $\left\{ \lambda_k \right\}$ was determined by the formula
\bb \label{stepsize} \ds \lambda_k = (1 - \rho c_k) \frac{\lambda_0}{\alpha k^s / P + 1},\ee
where the sequence $c_k$ starts with $c_0 = 0$ and the following terms are given by
\begin{equation*} \ds c_k = \frac{\left\langle \vetx^{k-1/2} - \vetx^{k-1}, \vetx^k - \vetx^{k-1/2} \right\rangle}{\left\|\vetx^{k-1/2} - \vetx^{k-1} \right\| \left\| \vetx^k - \vetx^{k-1/2} \right\|}.\end{equation*}
Each $c_k$ is the cosine of the angle between the directions taken by optimality and feasibility operators in the previous 
iteration. Thus, the factor $(1 - \rho c_k)$ serves as an empirical way to prevent oscillations. 
Finally, we use $\lambda_0 = \mu \left\|R \vetx^0 - \imb \right\|_1 / \left\| \vetg^0 \right\|^2$, where $\mu$ is the number of 
parcels in 
which the sum is divided and $\vetg^0$ is a subgradient of objective function in $\vetx^0$. Other free parameters in 
(\ref{stepsize}) were tuned and set to: $\rho = 0.999$, $s = 0.51$ and $\alpha = 1.0$.

\hspace{0.7cm} Now we need to calculate the subgradients for the objective function and $TV$. Since the vector 
\begin{equation*} \ds \mathbf{sign}(\vetx) = [u_i]^T, \quad \mbox{such that} \quad u_i := \left\{ \begin{array}{ccc} 1, & \mbox{if} & x_i > 0, \\ 0, & \mbox{if} & x_i = 0, \\ -1, &  & \mbox{otherwise} \end{array} \right. \end{equation*}
belongs to the set $\partial \left\| \vetx \right\|_1$, then the theorem 4.2.1, p. 263 in \cite{hil93} guarantees that
\begin{equation*} R^T \mathbf{sign}(R \vetx - \imb) \in \partial \left\| R \vetx - \imb \right\|_1,\end{equation*}
and this subgradient will be used in our experiments. In particular, for each $k \geq 0$, $\ell = 1, \dots, P$ and $s = 1, \dots, 
m(\ell)$ we use 
\begin{equation} \label{subgrad_objfunc}
\vetg_{i_{s}^{\ell}}^{k} =  \left\{ \begin{array}{ccc} \vetr_{i_{s}^{\ell}}, & \mbox{if} & \langle \vetr_{i_{s}^{\ell}}, \vetx_{i_{s-1}^{\ell}}^{k} \rangle - b_{i_{s}^{\ell}} > 0, \\
						       - \vetr_{i_{s}^{\ell}}, & \mbox{if} & \langle \vetr_{i_{s}^{\ell}}, \vetx_{i_{s-1}^{\ell}}^{k} \rangle - b_{i_{s}^{\ell}} < 0, \\
						       0, & & \mbox{otherwise.} \end{array} \right.
\end{equation}
A subgradient $\veth = [t_{i,j}]$ for $h(\vetx) = TV(\vetx) - \tau$ can be computed by
\begin{equation} \label{subgrad_tv} \begin{array}{rcl} \ds t_{i,j} & = & \ds \frac{2x_{i,j} - x_{i,j-1} - x_{i-1,j}}{\sqrt{(x_{i,j} - x_{i,j-1})^2 + (x_{i,j} - x_{i-,j})^2}} \\ & & 
  \ds \qquad + \frac{x_{i,j} - x_{i,j+1}}{\sqrt{(x_{i,j+1} - x_{i,j})^2 + (x_{i,j+1} - x_{i-1,j+1})^2}} \\ & & 
  \ds \qquad \qquad + \frac{x_{i,j} - x_{i+1,j}}{\sqrt{(x_{i+1,j} - x_{i,j})^2 + (x_{i+1,j} - x_{i+1,j-1})^2}},  \end{array} \end{equation}
where, if any denominator is zero, we annul the correspondent parcel.
\end{description}

Once we have determined the strings by setting $\Delta_1 = S_1, \dots, \Delta_P = S_P$ and weights ($w_\ell = 1/P$ for all $\ell$), 
the step-size sequence $\lambda_k$ in (\ref{stepsize}), initial image $\vetx^0$ in (\ref{initial_image}) and subgradients 
$\vetg_{i_{s}^{\ell}}^{k} \in \partial f_{i_{s}^{\ell}}(\vetx_{i_{s-1}^{\ell}}^{k})$ in (\ref{subgrad_objfunc}), 
optimality operator $\OO_f$ (\ref{initial_opt_op})-(\ref{SA_isoo}) can be applied. 
By considering the subdifferential of $\| \vetx \|_1$, it is clear that $\partial f(\vetx)$ is uniformly bounded, ensuring that 
assumption (\ref{limitacao}) holds. Moreover, since $\rho \in [0,1)$, $\alpha > 0$, $s \in (0, 1]$ and $c_k \in [-1,1]$, by 
Cauchy-Schwarz inequality we can ensure that $\lambda_k > 0$ and that Condition \ref{cond3} of Theorem \ref{conv1} holds. 

Using $\veth \in \partial h(\vetx)$ defined in 
(\ref{subgrad_tv}), operator $\SSS_{h}^{\nu}$ can be computed by equation (\ref{subgrad_projection}), such that, 
in our tests, we use $\nu = 1$. 
The feasibility operator is thus given by $\ds \V_{\X} = \mathcal{P}_{\mathbb{R}_{+}^{n}} \circ \SSS_{h}^{\nu}$ 
(see equation (\ref{viabilidade})). The projection step can be regarded as a special case of the operator $\SSS_{g}^{\nu}$ with 
$\nu=1$ and $g = d_{\mathbb{R}_{+}^{n}}$. It is easy to see that $\V_{\X}$ defined in this way satisfies the conditions 
established in the Proposition \ref{prop viab}. 

In conclusion, once $\left\| R \vetx - \imb \right\|_1 \geq 0$, Corollary 2.7 in \cite{helou09} implies that 
$\left\{d_{\X}(\vetx^k)\right\} \rightarrow 0$ (the sequence is bounded). Since $\partial f(\vetx)$ is uniformly bounded, we 
have that $\left[ f(\PX(\vetx^k)) - f(\vetx^k) \right]_{+} \rightarrow 0$ and Condition \ref{cond4} of Theorem \ref{conv1} holds.
Thus, Corollary \ref{convSA} can be applied ensuring convergence of the Algorithm \ref{alg-proposed}.

\subsection{Image reconstruction analysis}
To run the experiments, we used a computer with processor Intel Core i7-4790 CPU @ 3.60 GHz x8 and 16 GB of memory. 
The operating system used was Linux and the implementation was realized in C$++$.
Figure \ref{graficos} shows the decrease of the objective function with respect to computation time to compare convergence 
speed and image quality in the performed reconstructions. Furthermore, in order to obtain a more meaningful analysis, 
we consider the graphic of the total variation $TV(\vetx)$ and the \textit{relative squared error},
\begin{equation*} \ds RSE(\vetx) = \frac{\left\|\vetx - \vetx^{\ast} \right\|^{2}}{\left\| \vetx^{\ast} \right\|^2}.\end{equation*}
Note that the $RSE$ metric requires information on the desired image. 
Also we show graphs of $TV(\vetx^k)$ as function of $f(\vetx^k)$.
\begin{figure}
\centering
\includegraphics[scale=0.8]{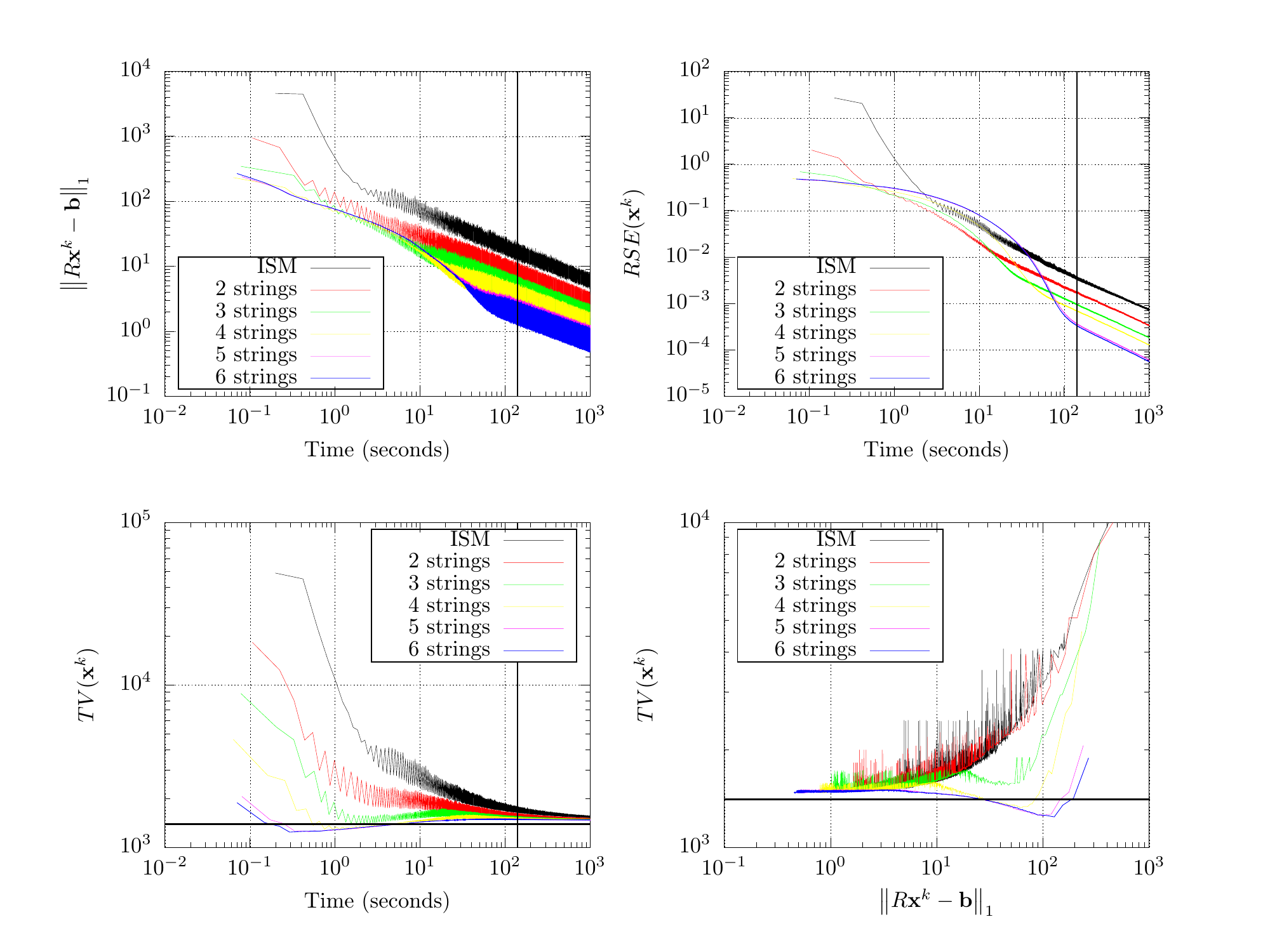}
\caption{Decrease of the objective function, $TV$ and $RSE$ in a noise-free condition. 
Comparing ISM (1 string) with algorithms that use 2-6 strings, executed in parallel, 
lower values are reached for the objective and $RSE$ functions. 
For $TV$, we get an oscillation with lower intensity (especially when we used 4-6 strings). 
Note that solid horizontal lines on $TV$ graph represent the target value $\tau = TV(\vetx^{\ast})$. 
The solid vertical lines show, for a fixed computation time, that both functions values are decreasing with respect to the 
number of strings $P$. 
Figure \ref{imagens} shows the reconstructed images by the algorithms for this fixed computation time. 
In the bottom right figure, note that $TV$ values appear to represent, at many fixed levels of residual $\ell_1$-norm, 
a decreasing function of the number of strings.} \label{graficos}
\end{figure}

When we use $P = 2, \dots, 6$ (algorithm is executed with 2-6 strings), it is possible to observe a faster decrease in the 
values of the objective function $f(\vetx^k)$ as the running time increases, if compared to the case where $P=1$. 
Since there is no guarantee that algorithm produces a descent direction in each iteration, it is important to note that, 
in some of the tests, the intensity of the oscillation, i.e., $f(\vetx^{k+1}) - f(\vetx^k)$ for $f(\vetx^{k+1}) > f(\vetx^{k})$, 
decreases as the number of strings 
increases (note, for example, the algorithm performance with 5 and 6 strings, from $0$ to about $40$ seconds). A similar behavior 
can be noted for values $TV(\vetx^k)$.
For 4-6 strings, the algorithm is able to provide images with a more appropriate $TV$ level, approaching the feasible region more quickly.
Even if closer to satisfying the constraints, for methods with a larger number of strings,
the values of $RSE(\vetx^k)$ and of the objective function decrease with lower intensity oscillations 
and reach lower values within the same computation time. Interestingly, the experiments with noisy data show that 
the algorithm generates a sharp decrease in the intensity of oscillation precisely where image quality seems to reach a good level.
The study of conditions under which we can establish a stopping criterion for the algorithm are left for future research, perhaps taking advantage of this kind of phenomenon.

The quality of the reconstruction is significantly affected by the increase in the number of strings. 
Figure \ref{imagens} shows the reconstructed images obtained in the experiments. 
There is a clear difference in the quality of reconstruction for ISM and algorithms with 2-6 strings. 
For 5 and 6 strings the reconstruction is visually perfect.

\begin{figure}[htbp!]
\centering
\subfloat[ISM (1 string)]{
\includegraphics[scale=0.4]{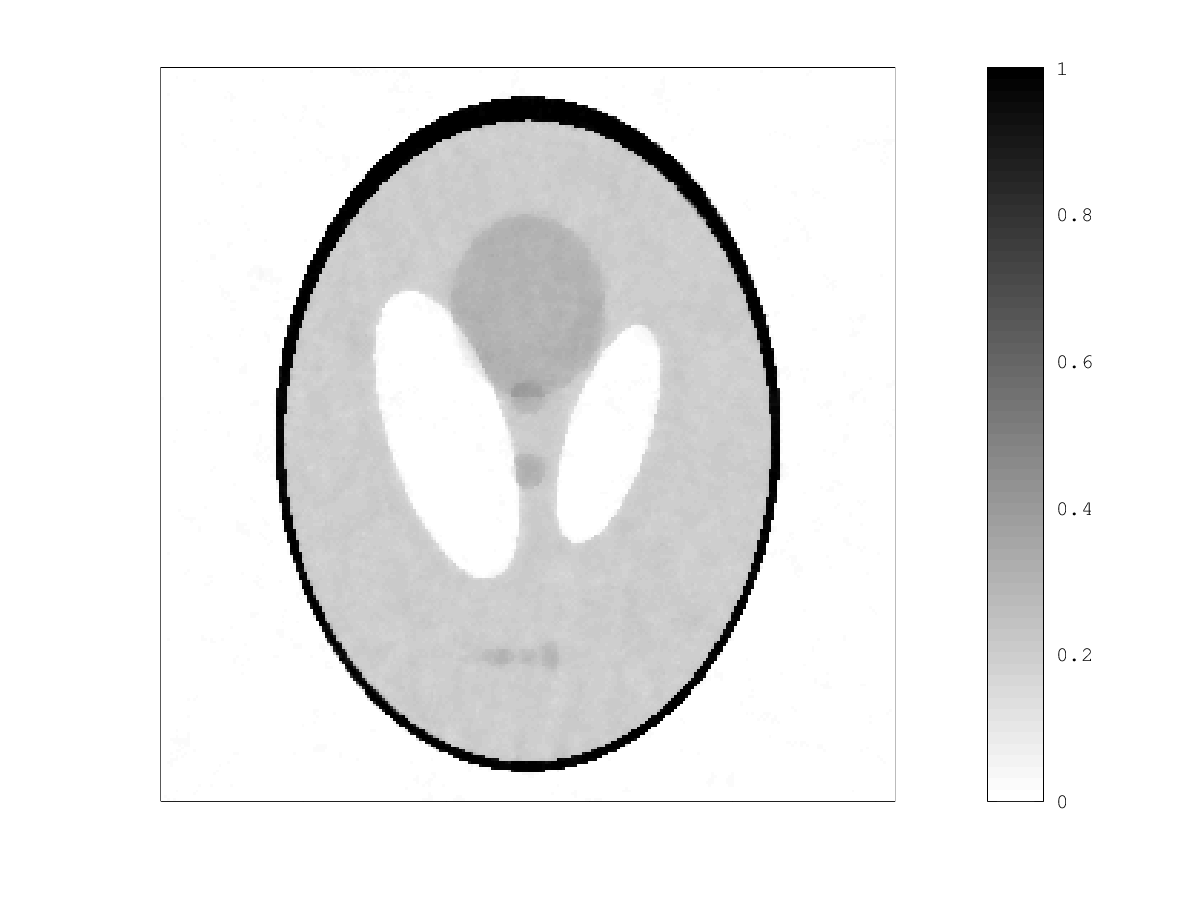}
}
\subfloat[2 strings]{
\includegraphics[scale=0.4]{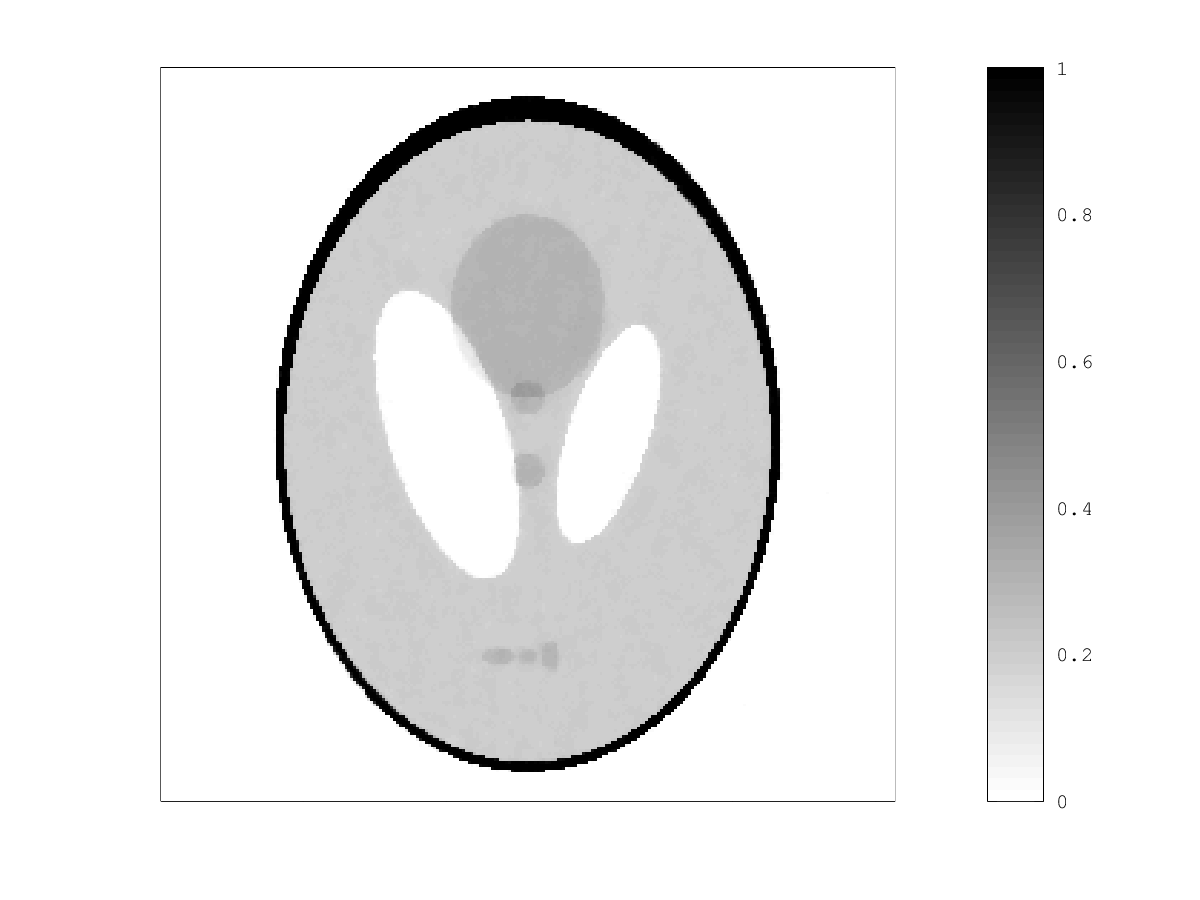}
}
\\ 
\subfloat[3 strings]{
\includegraphics[scale=0.4]{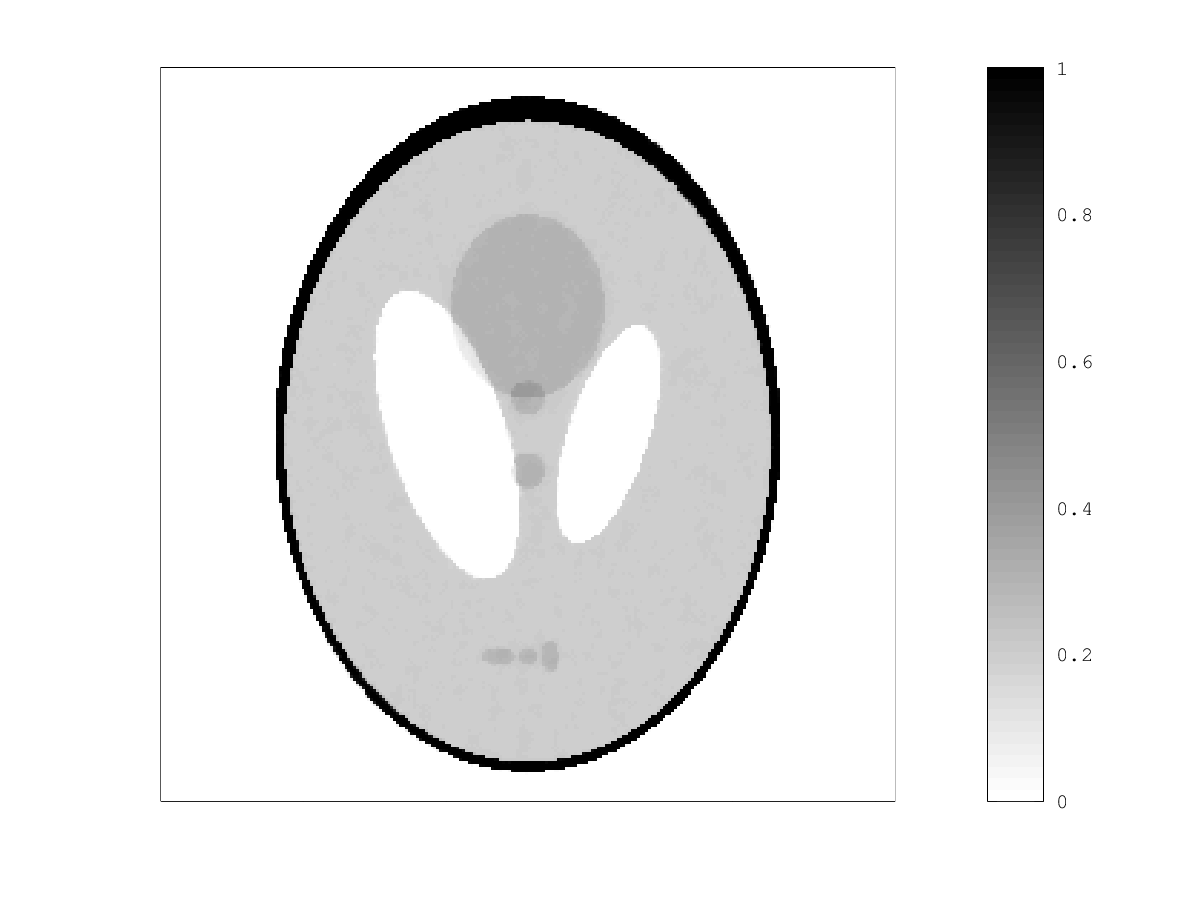}
}
\subfloat[4 strings]{
\includegraphics[scale=0.4]{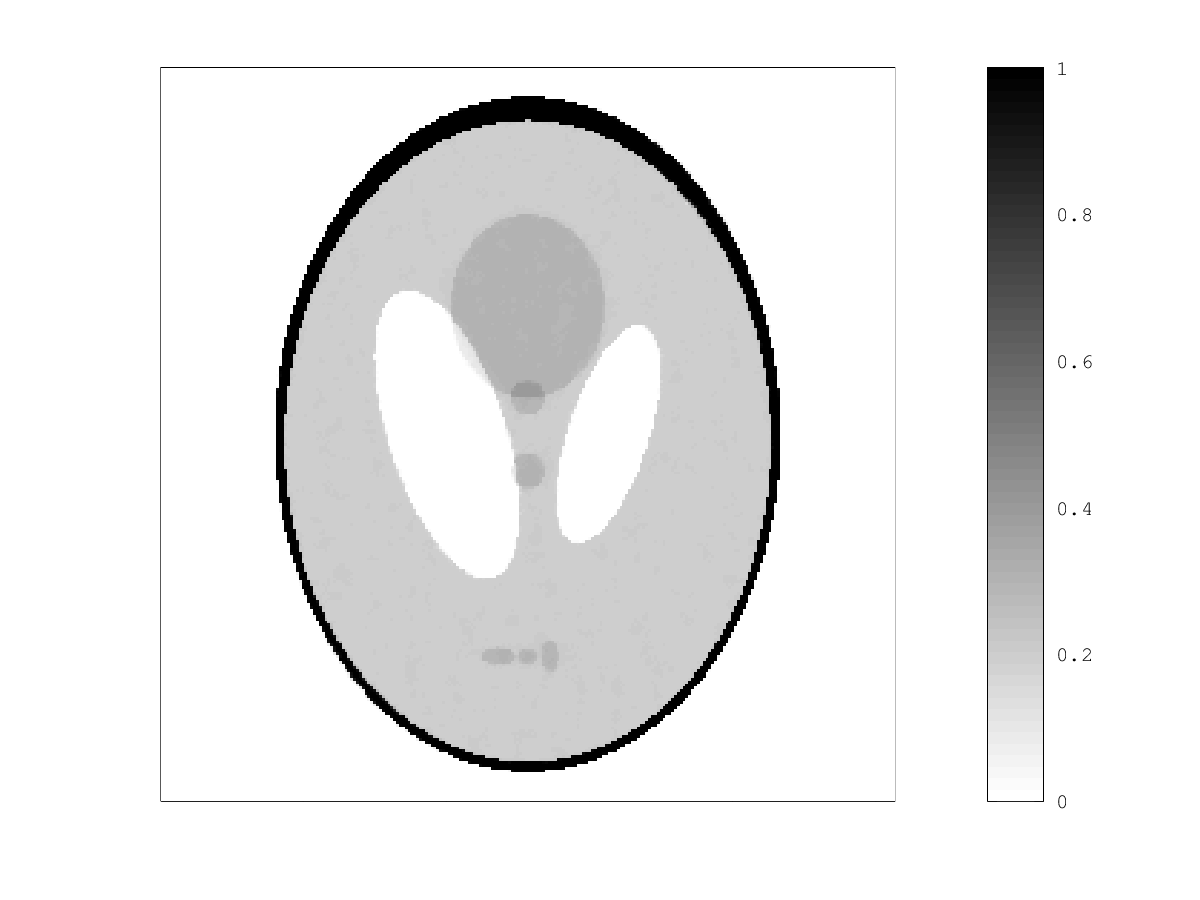}
}
\\
\subfloat[5 strings]{
\includegraphics[scale=0.4]{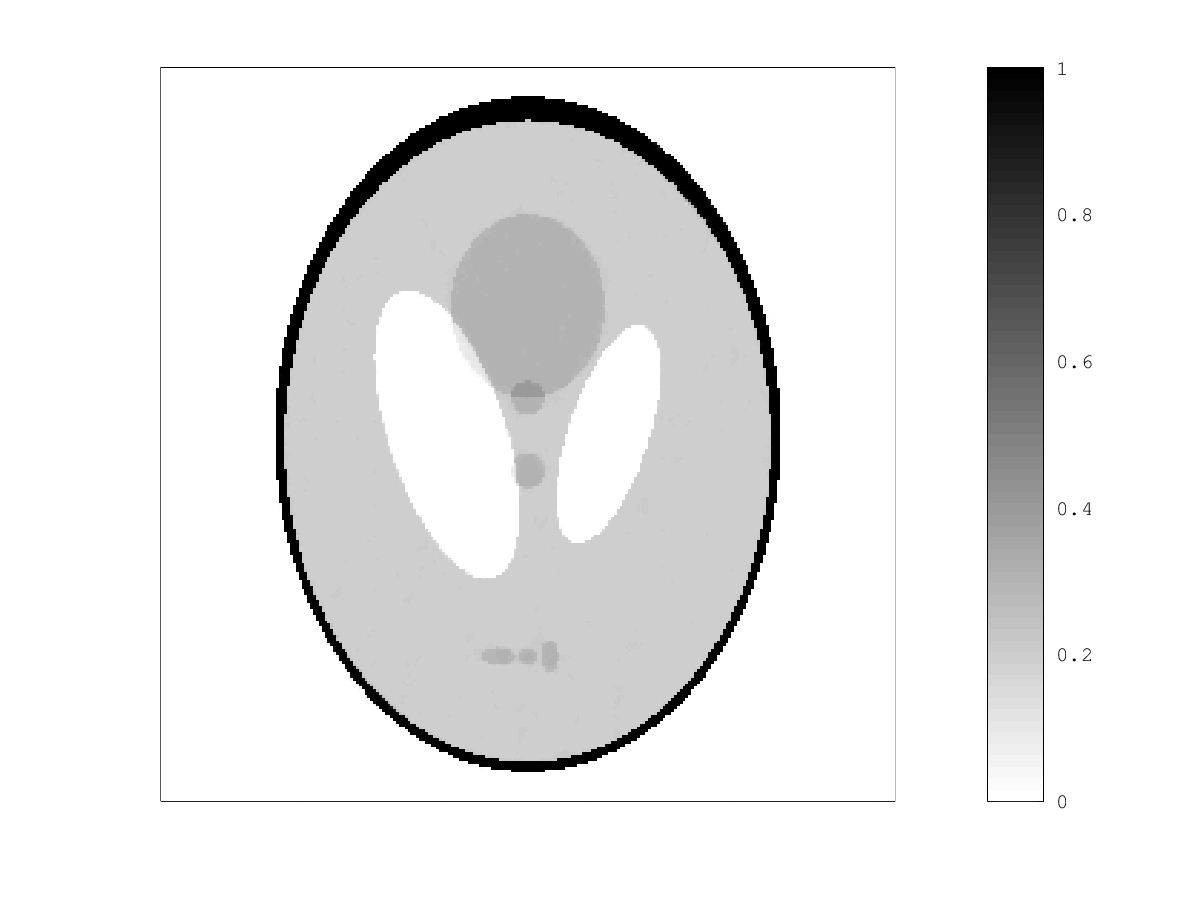}
}
\subfloat[6 strings]{
\includegraphics[scale=0.4]{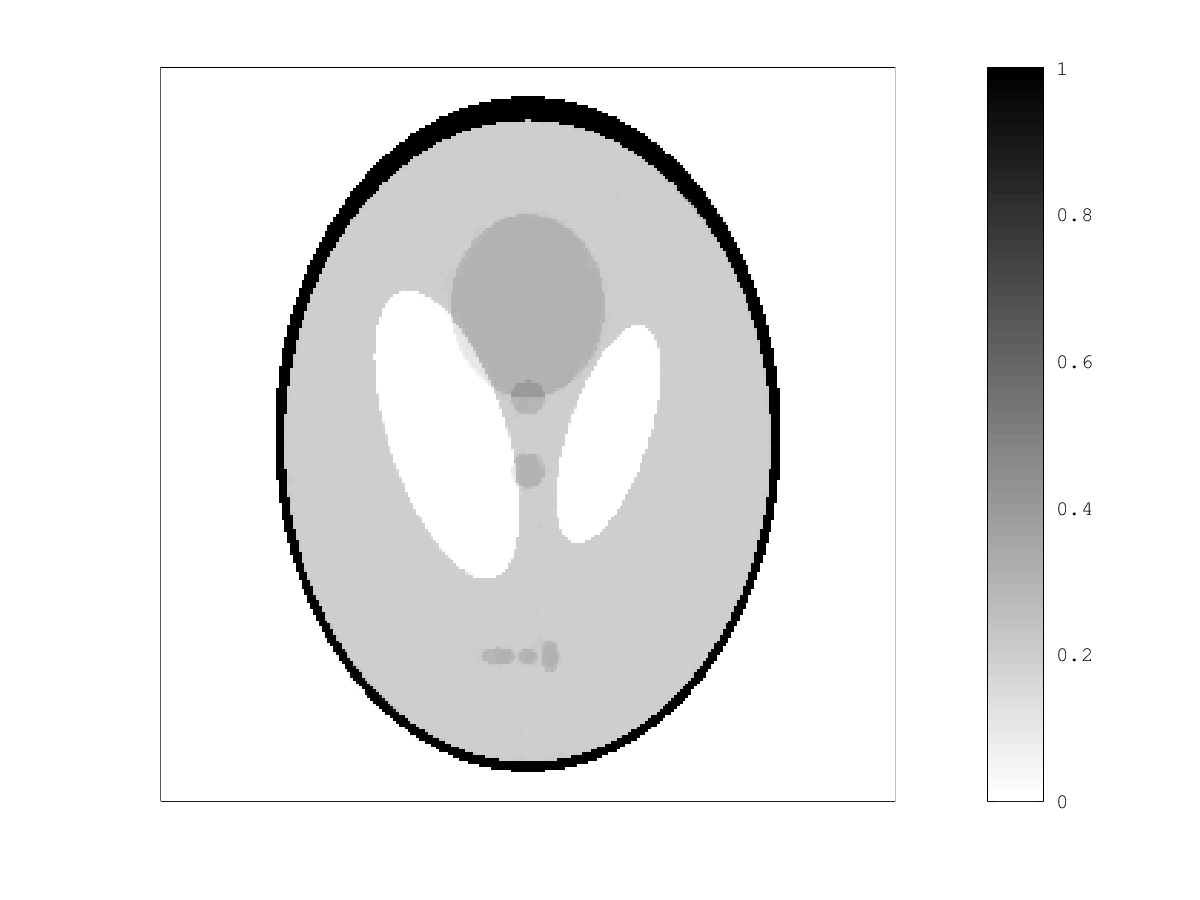}
}
\caption{Reconstructed images obtained in the computation time mentioned in the Figure \ref{graficos}.} \label{imagens}
\end{figure}

Figures \ref{ruido1}, \ref{ruido3} and \ref{ruido2} show plots similar to those in Figure~\ref{graficos} but now under different relative noise levels, which was computed as $\left\| \imb - \imb^{\dagger} \right\| / \left\| \imb^{\dagger} \right\|,$ 
where $\imb^{\dagger}$ is the vector that contains the ideal data.
\begin{figure}[!ht]
\centering
\includegraphics[scale=0.8]{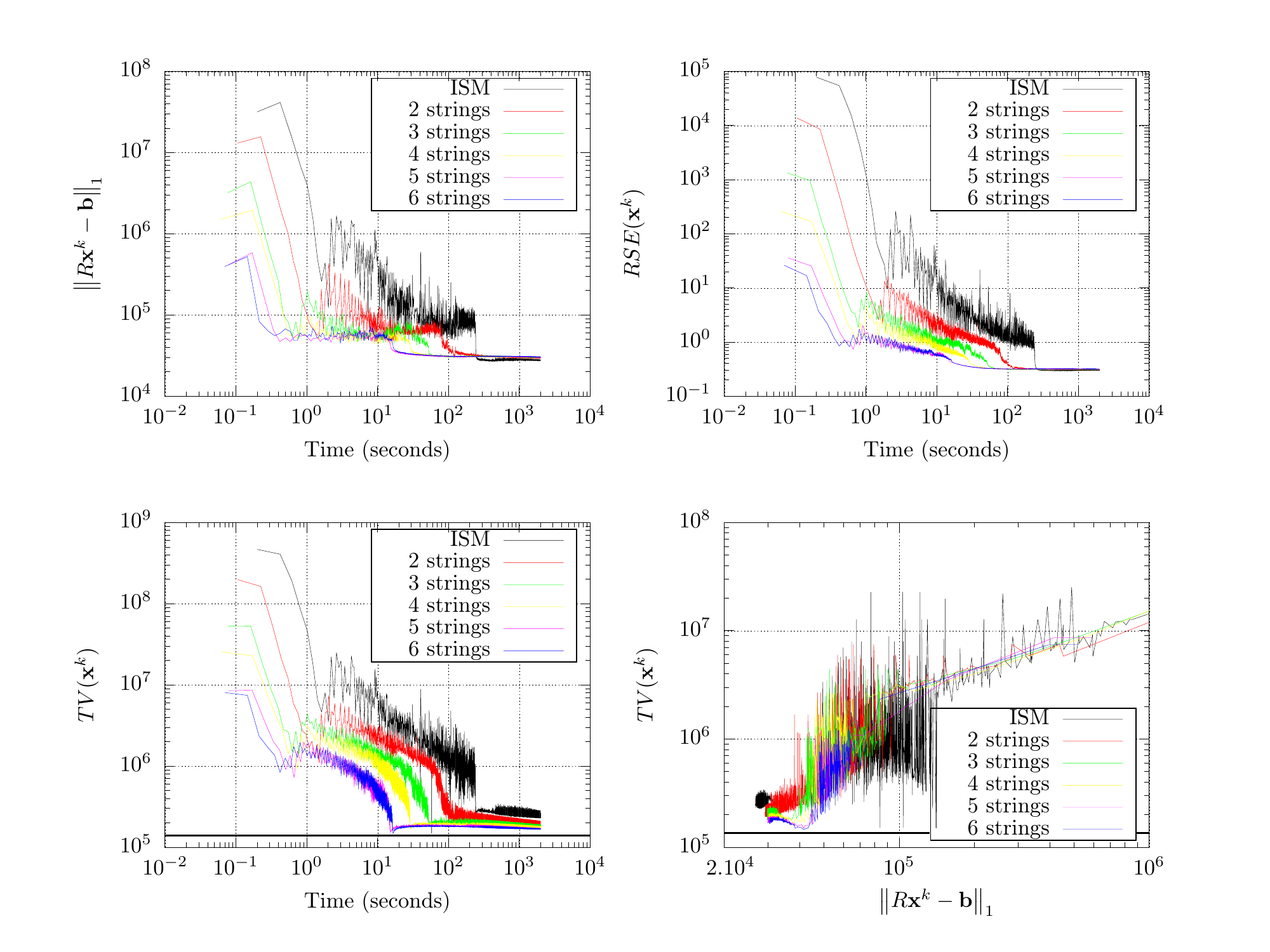}
\caption{Test with $17.8\%$ of relative noise.}\label{ruido1}
\end{figure}
\begin{figure}[!ht]
\centering
\includegraphics[scale=0.8]{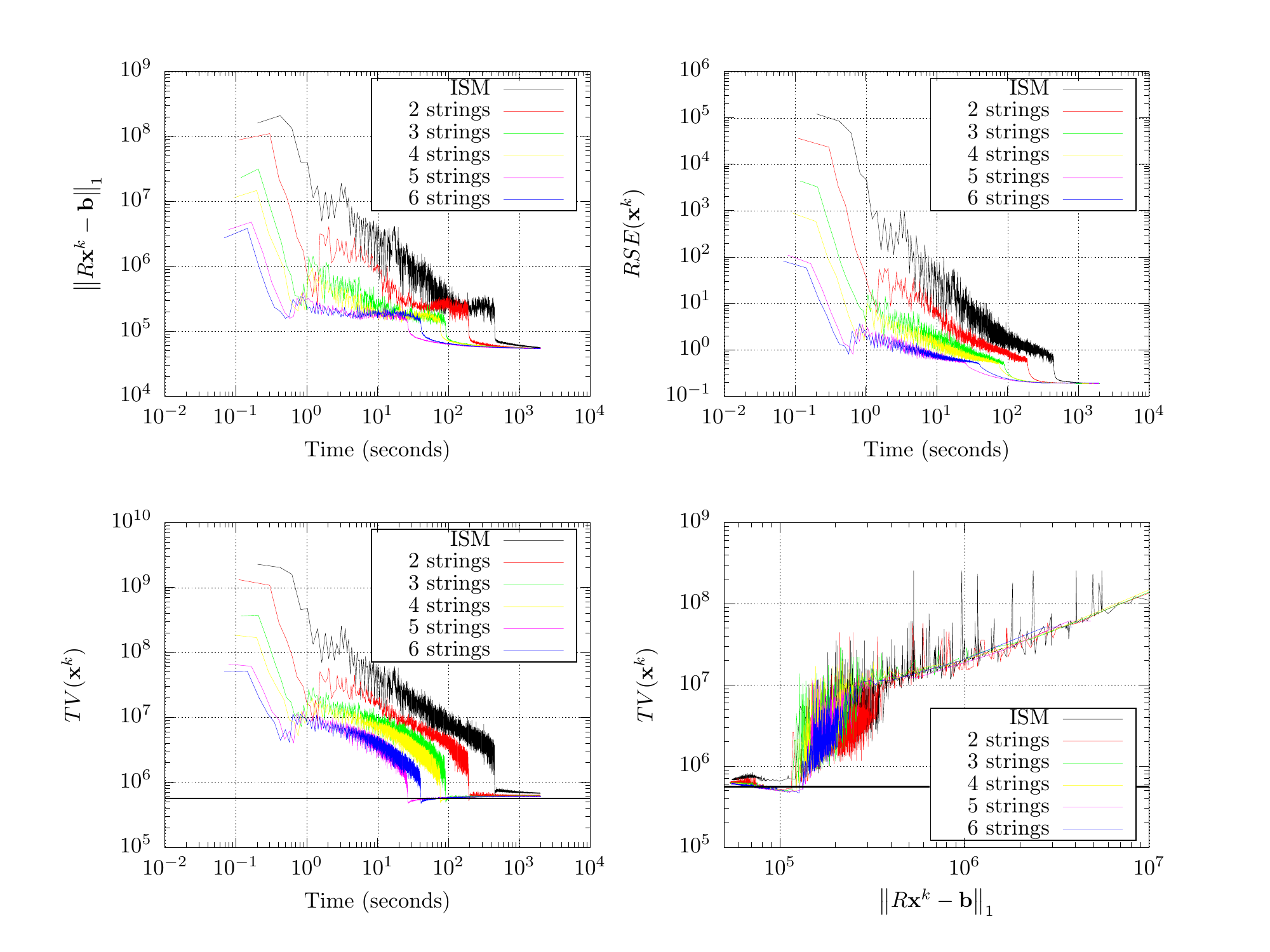}
\caption{Test with $8.78\%$ of relative noise.}\label{ruido3}
\end{figure}
\begin{figure}[!ht]
\centering
\includegraphics[scale=0.8]{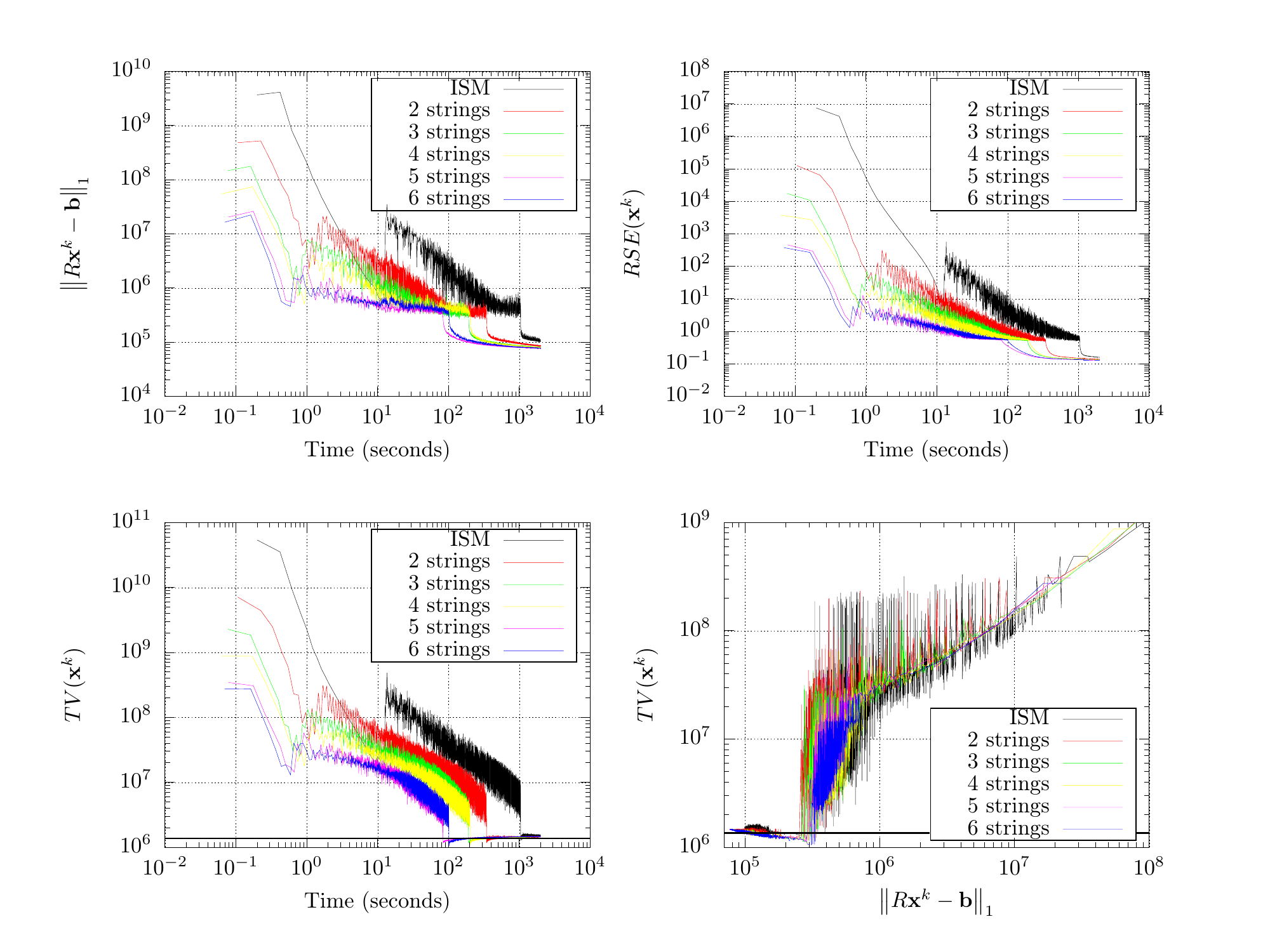}
\caption{Test with $5.65\%$ of relative noise.}\label{ruido2}
\end{figure}
We can note that the behavior of algorithm is similar to the previous case. 
Algorithms with larger number of strings reach results that the ISM takes longer to reach. 
Furthermore, oscillations with lower intensity can be noted, especially for 5 and 6 strings. 
Figure \ref{imagens_ruido} shows the reconstructed images obtained by ISM and method with $6$ strings according to the following rule: we set 
an objective function value and seek the first iteration to fall below this threshold for each method.
Table \ref{tab1} provides the running time and total variation for each case.
These data confirm a good performance of the algorithm with string averaging, in the sense that, for fixed values of objective function, 
algorithm running with $6$ strings provides images in which quality appears to be improved (or at least is similar) with lower running time, if compared against ISM.

\begin{table}[htbp!]
\centering
\setlength{\arrayrulewidth}{2\arrayrulewidth}
\setlength{\belowcaptionskip}{10pt}
\begin{tabular}{|c|c|c|}
\hline
 Method / Noise / $f(\vetx)$  & Time ($s$) & $TV(\vetx)$ \\
\hline \hline
ISM / $17.8\%$ / $3.194\times10^4$ &  $2.45\times10^2$ & $2.7\times10^5$ \\
\hline
$P=6$ / $17.8\%$ / $3.191\times10^4$ & $6.0\times10^1$ & $1.82\times10^5$ \\
\hline
ISM / $8.78\%$ / $6.070\times10^4$ & $7.11\times10^2$ & $6.91\times10^5$ \\
\hline
$P=6$ /$8.78\%$ / $6.093\times10^4$ & $1.5\times10^2$ & $5.87\times10^5$ \\
\hline
ISM / $5.65\%$ / $9.889\times10^4$ & $1.87\times10^3$ & $1.48\times10^6$ \\
\hline
$P=6$ / $5.65\%$ / $9.889\times10^4$ & $2.2\times10^2$ & $1.36\times10^6$ \\
\hline
\end{tabular}
\caption{Running time and total variation obtained by ISM and method with $6$ strings under conditions of Poisson relative noise used in the 
tests for some fixed values of objective function.} \label{tab1}
\end{table}

\begin{figure}[htbp!]
\centering
\subfloat[ISM / $17.8\%$ / $3.194\times10^4$.]{
\includegraphics[scale=0.4]{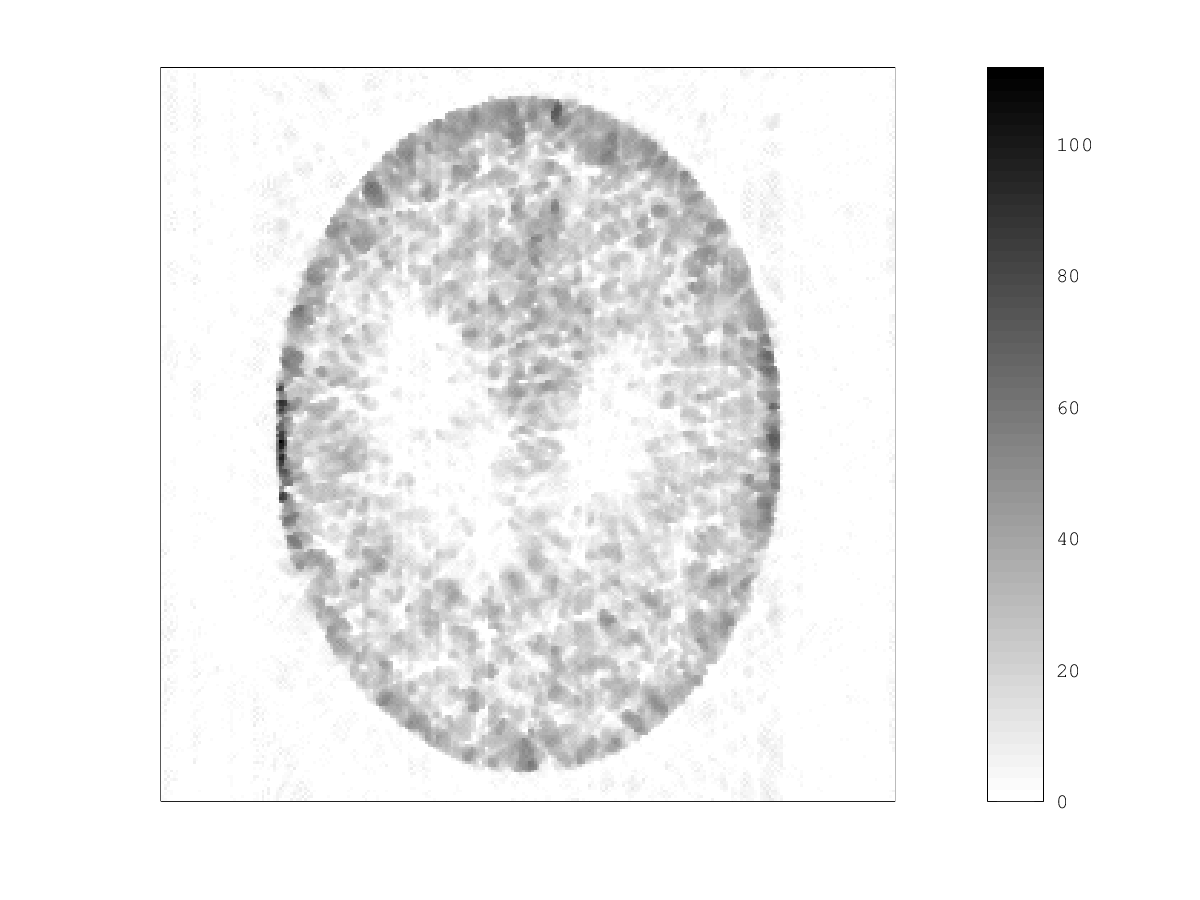}
}
\subfloat[$P=6$ / $17.8\%$ / $3.191\times10^4$.]{
\includegraphics[scale=0.4]{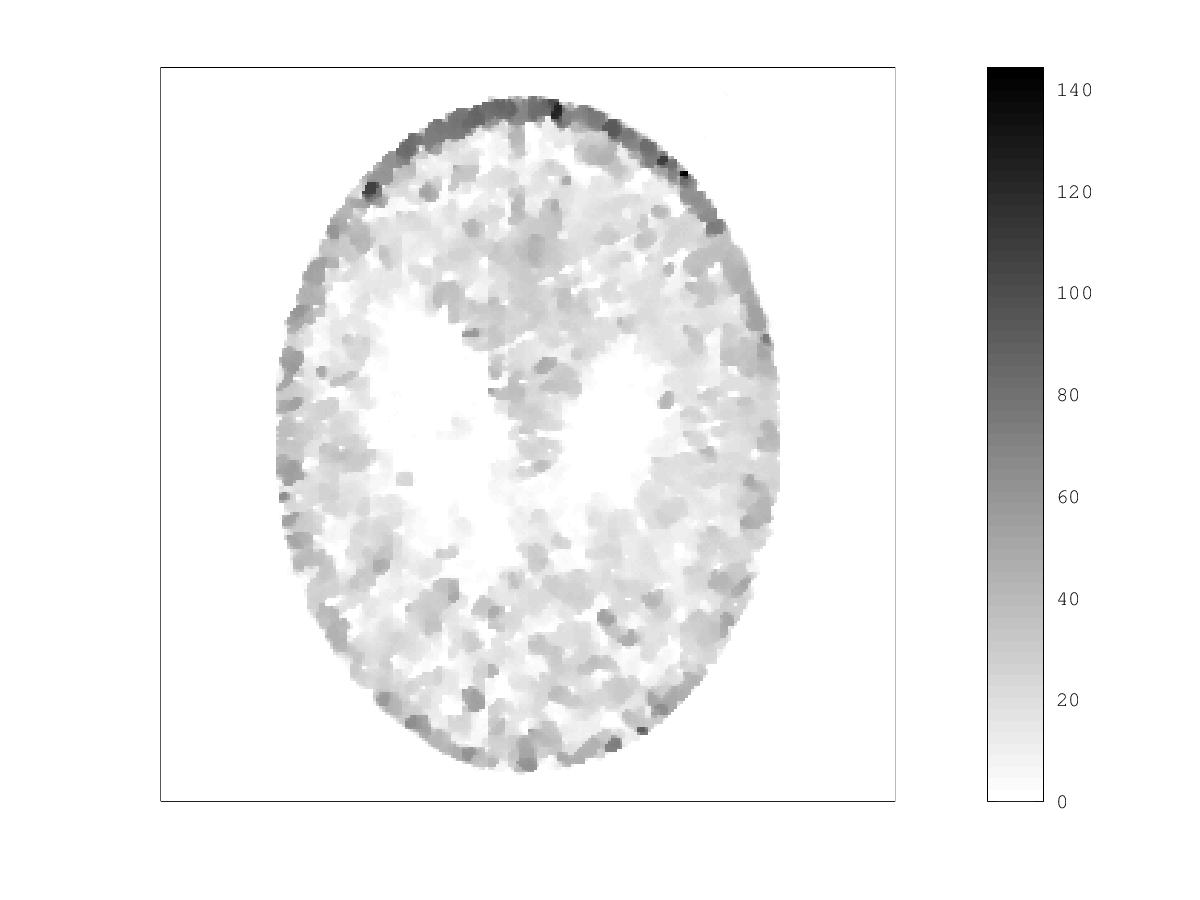}
}
\\ 
\subfloat[ISM / $8.78\%$ / $6.070\times10^4$]{
\includegraphics[scale=0.4]{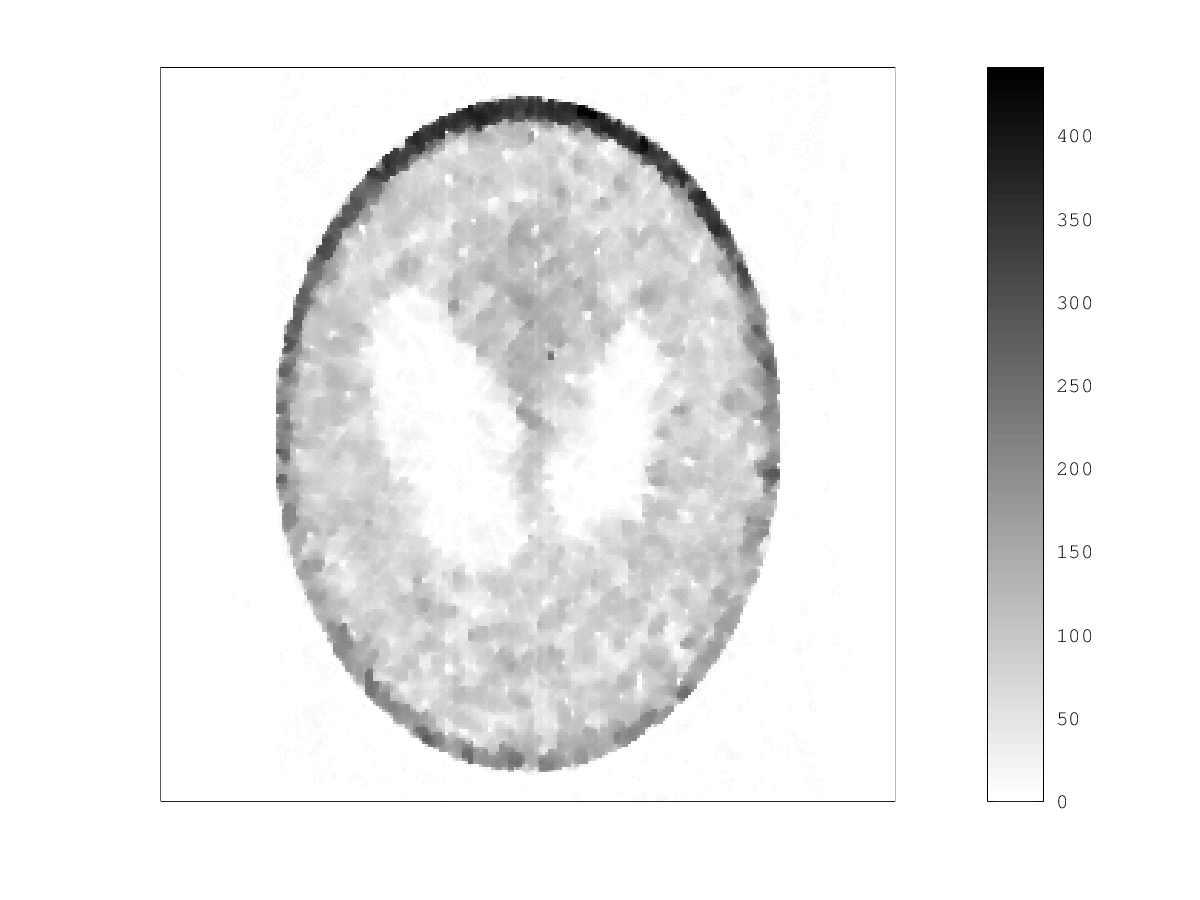}
}
\subfloat[$P=6$ / $8.78\%$ / $6.093\times10^4$.]{
\includegraphics[scale=0.4]{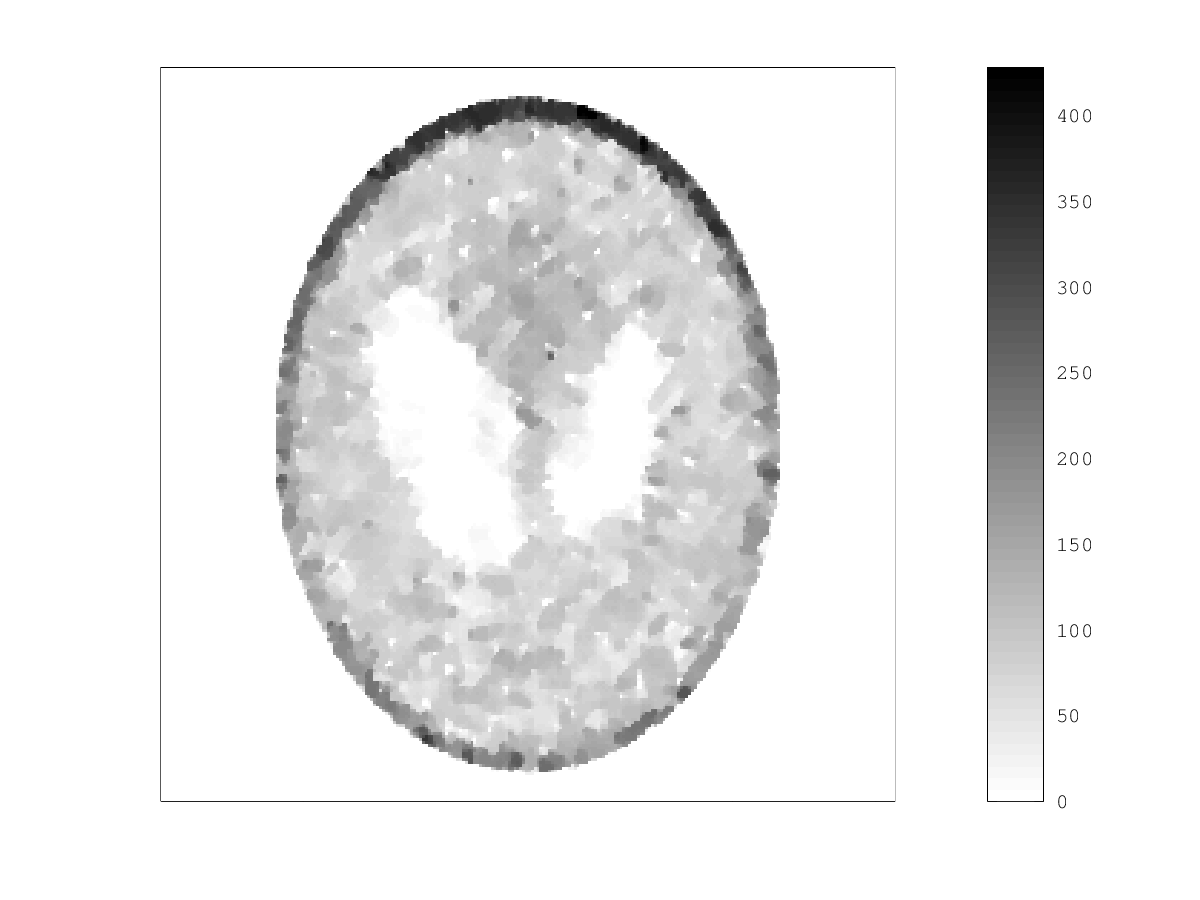}
}
\\
\subfloat[ISM / $5.65\%$ / $9.889\times10^4$.]{
\includegraphics[scale=0.4]{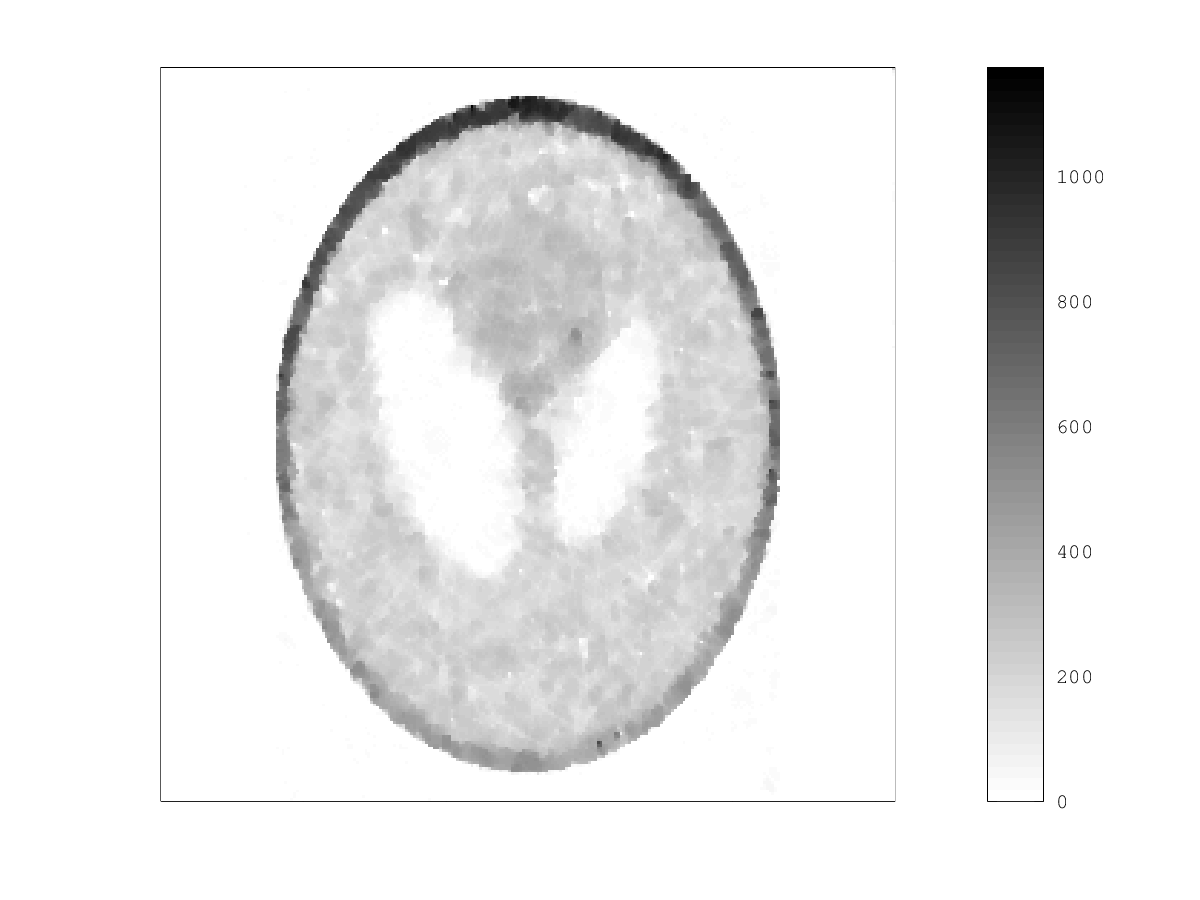}
}
\subfloat[$P=6$ / $5.65\%$ / $9.889\times10^4$.]{
\includegraphics[scale=0.4]{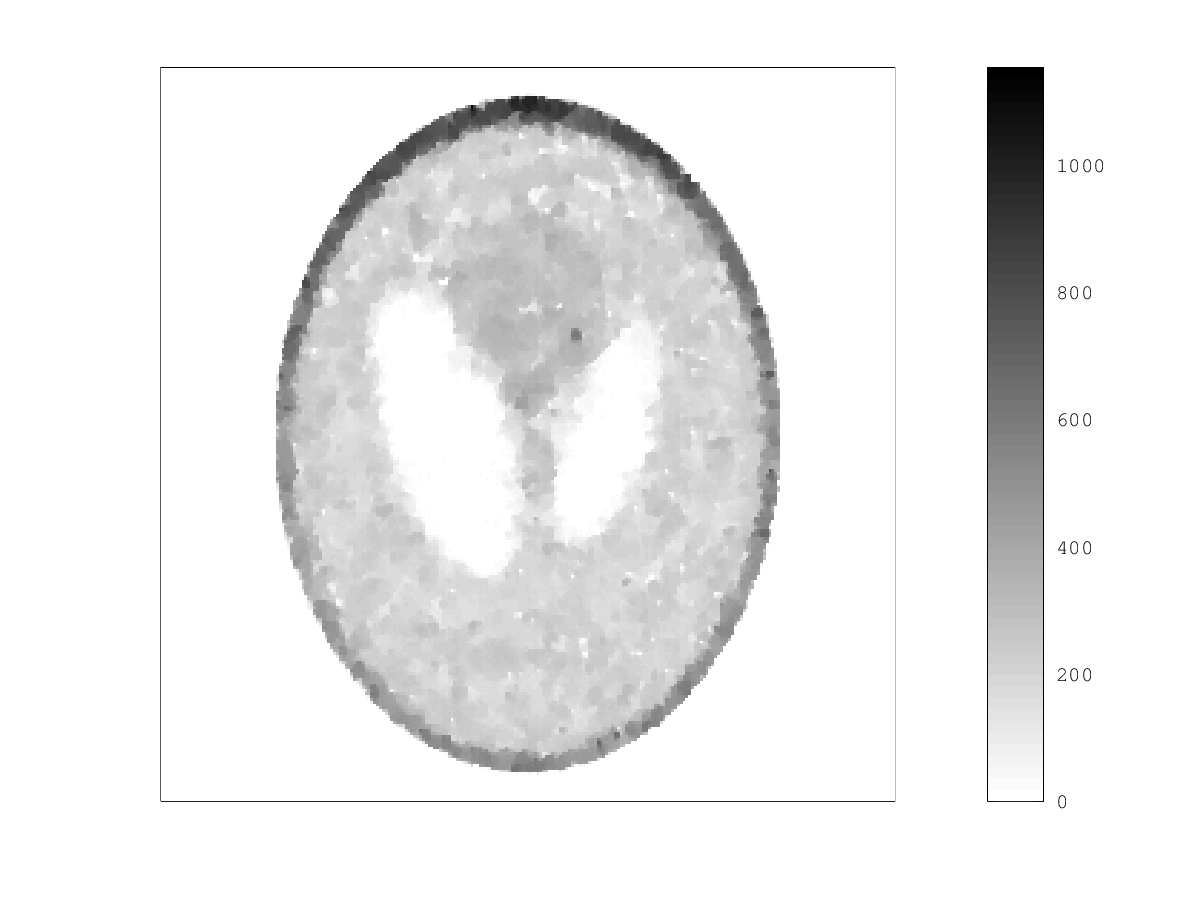}
}
\caption{Reconstructed images obtained in the tests with Poisson noise. Items (a)-(f) exhibit: method / relative noise / $f(\vetx)$. 
Table \ref{tab1} shows running time and total variation obtained in each case.} \label{imagens_ruido}
\end{figure}

\subsection{Tests with real data}

Tomographic data was obtained by synchrotron radiation illumination of eggs of fishes of the species \emph{Prochilodus lineatus} 
collected at the Madeira river's bed at Brazilian National Synchrotron Light Laboratory's (LNLS) facility. 
The eggs had been previously embedded in formaldehyde in order to prevent decay, but were later fixed in water within a 
borosilicate capillary tube for the scan. The sample was placed between the x-ray source and a photomultiplier coupled to a CCD 
capable of recording the images. 
After each radiographic measurement, the sample was rotated by a fixed amount and a new measurement was made.

A monochromator was added to the experiment to filter out low energy photons and avoid overheating of the soft eggs and the 
embedding water. This leads to a low photon flux, which increased the required exposure time to $20$ seconds for each projection 
measurement. Each of this radiographic image was $2048 \times 2048$ pixels covering an area of $0.76 \times 0.76$mm${}^2$. 
Given the high measurement duration of each projection, for the experiment to have a reasonable time span, we have collected 
only $200$ views in the $180^{\circ}$ range, leading to slice tomographic datasets (sinograms) each of dimension 
$2048 \times 200$ (see Figure \ref{ovas_sinograma}).
\begin{figure}[!h]
\centering
\includegraphics[scale=0.7]{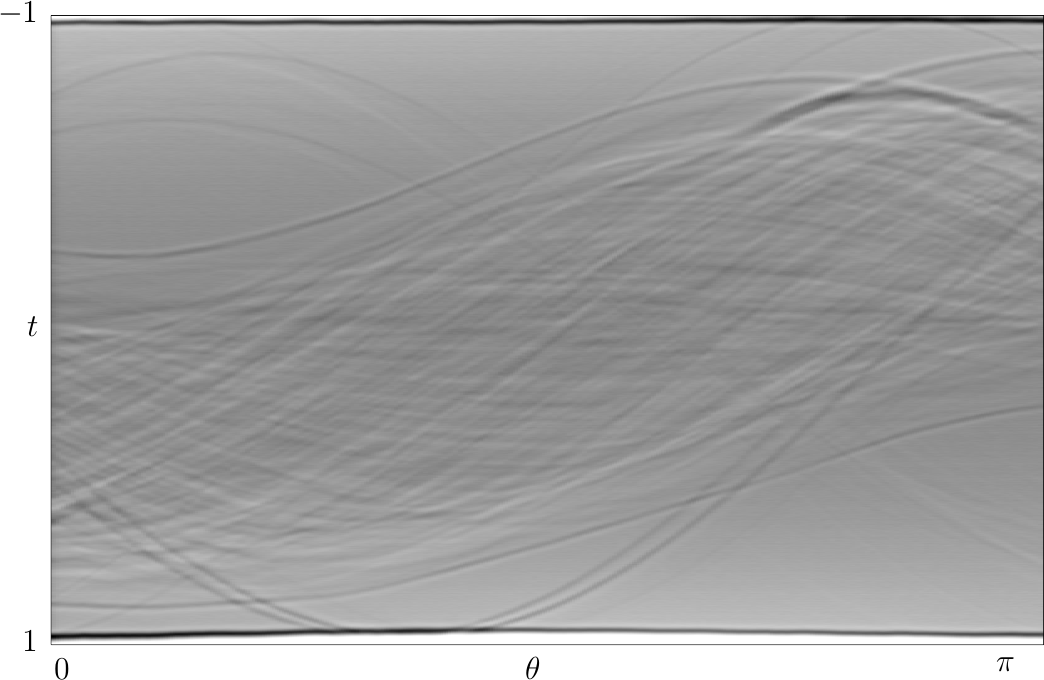}
\caption{Sinogram obtained by synchrotron radiation illumination of eggs of fishes.} \label{ovas_sinograma}
\end{figure}

In this experiment, we use $\tau = 5 \times 10^4$, $\nu = 1.5$ and the other parameters $\rho$, $\alpha$ and $s$ were the same as in the 
previous experiment. 
Furthermore, to avoid high step-size values, we multiply the initial step-size $\lambda_0$ by $0.25$.
Figure \ref{ovas1} shows the plot of $TV$ as function of residual $\ell_1$-norm. 
Better quality reconstructions are generated by the algorithm that uses 6 strings. 
Figure \ref{ovas_imagens} shows the images obtained by reconstruction. By considering that the eggs were immersed in water, which has homogeneous attenuation value, we can 
conclude that image 
in Figure \ref{ovas_imagens}-(b) has less artifacts. Figure \ref{profile} shows profile lines of the reconstructed images in Figure \ref{ovas_imagens}. Note that algorithm running with 
$6$ strings presents a reconstruction with less overshoot and more smoothness than ISM.

\begin{figure}[ht]
\centering
\includegraphics[width=0.825\textwidth]{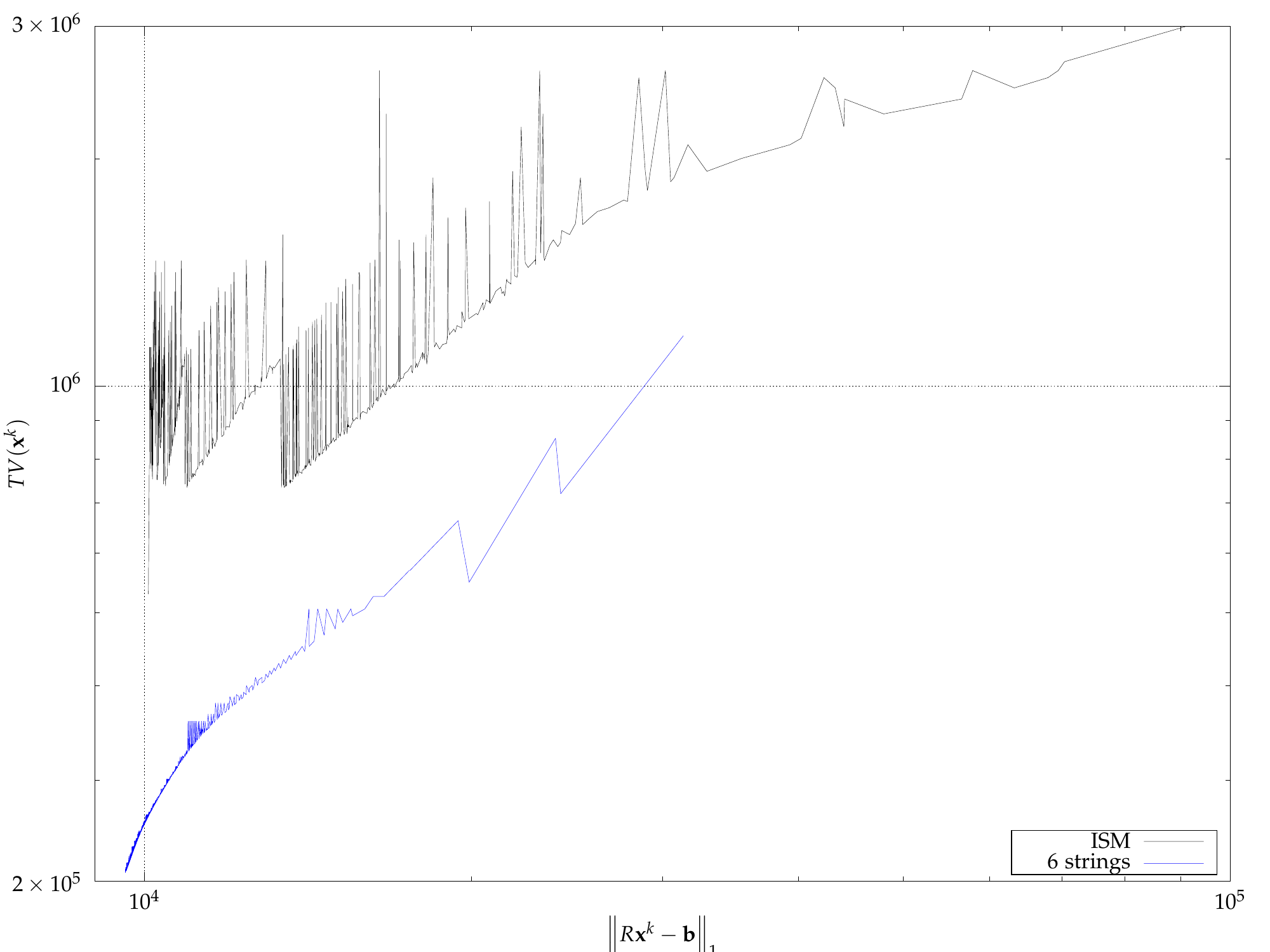}
\caption{Total variation as function of the residual $\ell_1$-norm for the experiment using eggs of fishes. 
Note that the method with 6 strings provides lower values for total variation with a lower oscillation level.} \label{ovas1}
\end{figure}

\begin{figure}[ht]
\begin{center}
\subfloat[ISM (1 string)]{
\includegraphics[width=0.46\textwidth,height=0.46\textwidth]{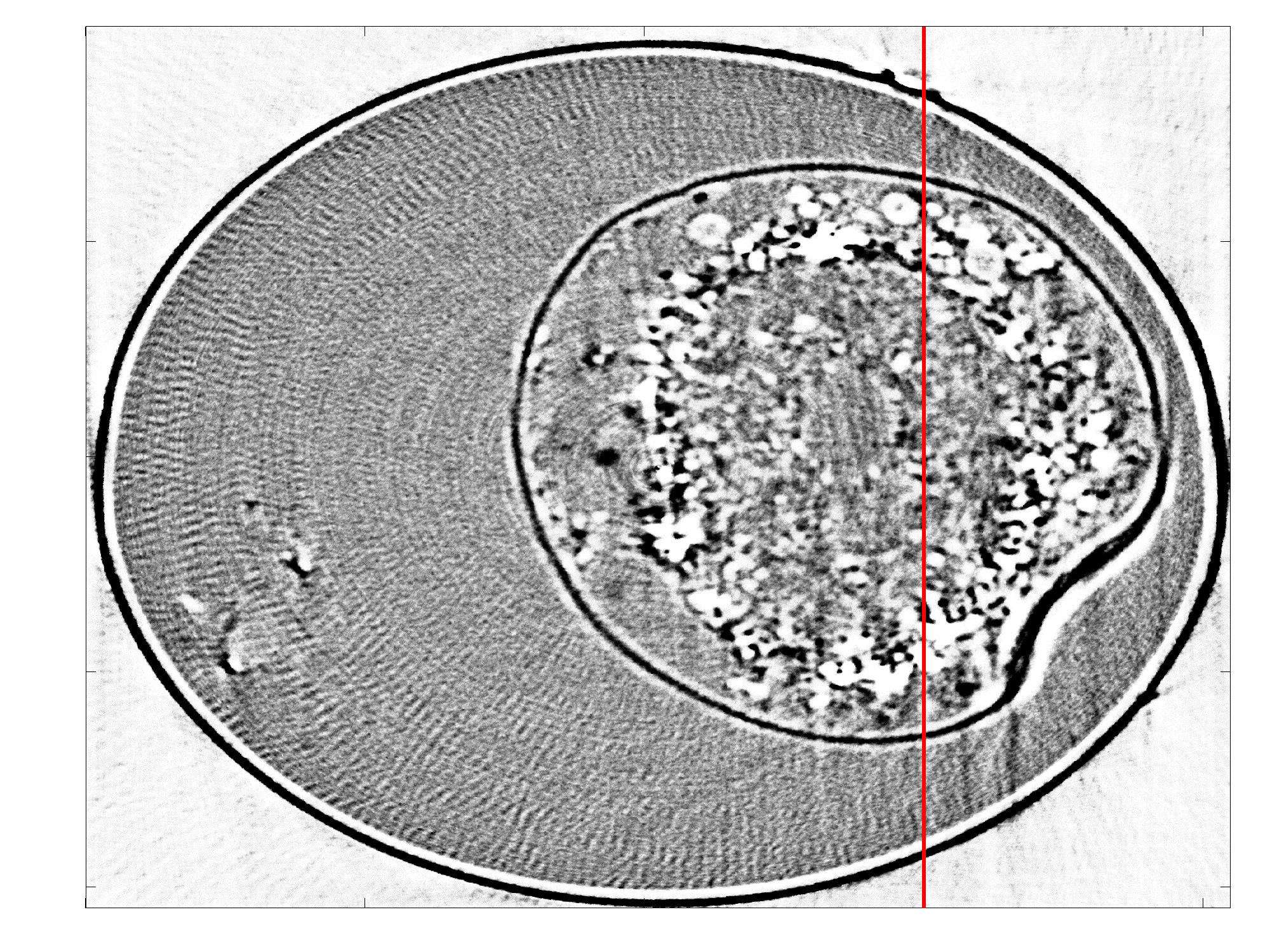}
}
\subfloat[6 strings]{
\includegraphics[width=0.46\textwidth,height=0.46\textwidth]{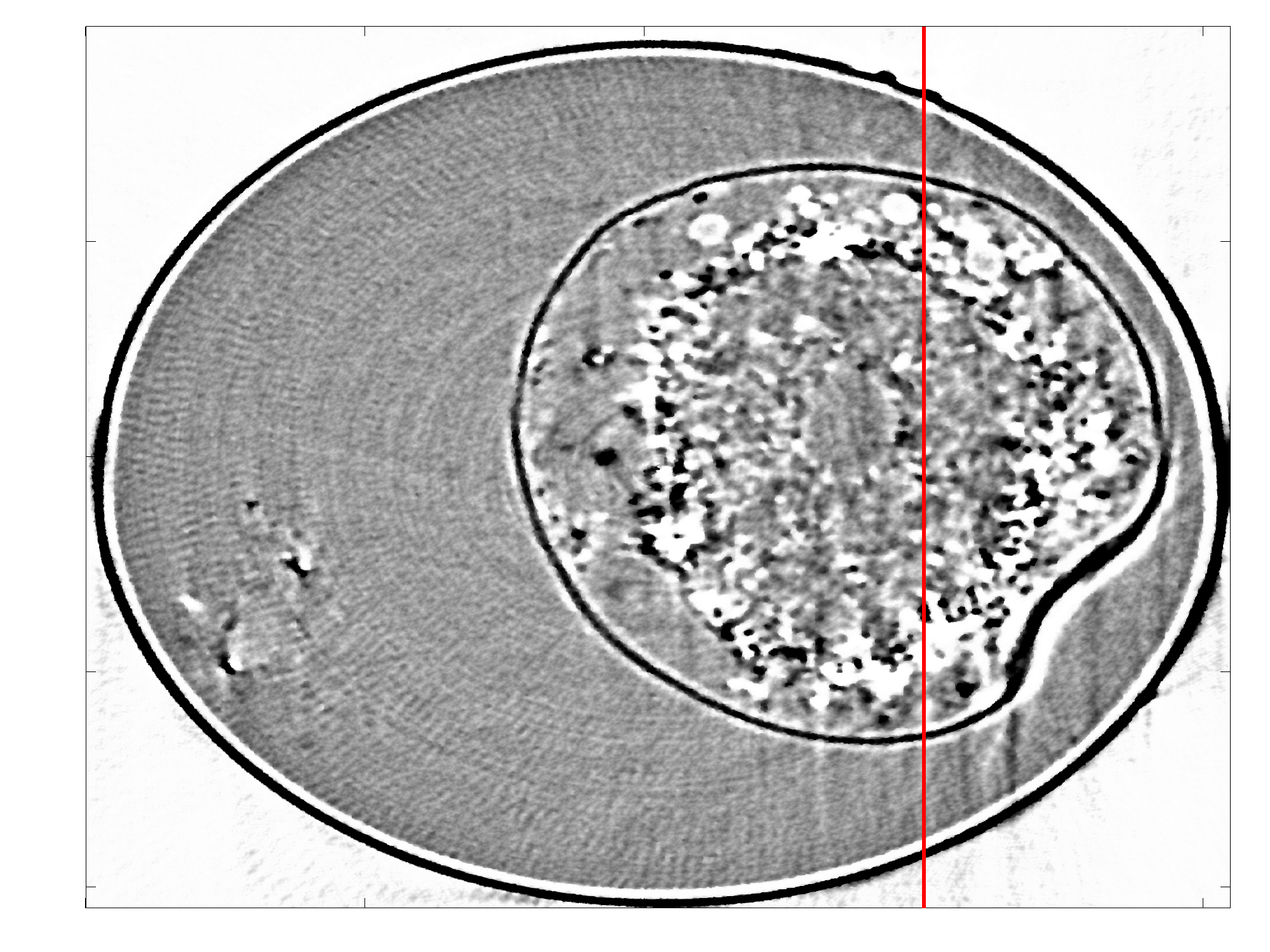}
}
\end{center}
\caption{Reconstructed images from sinogram given in Figure \ref{ovas_sinograma}. The vertical solid red lines show where the profiles of Figure \ref{profile} were taken 
from.} \label{ovas_imagens}
\end{figure}

\begin{figure}[ht!]
\centering
\includegraphics[width=0.825\textwidth]{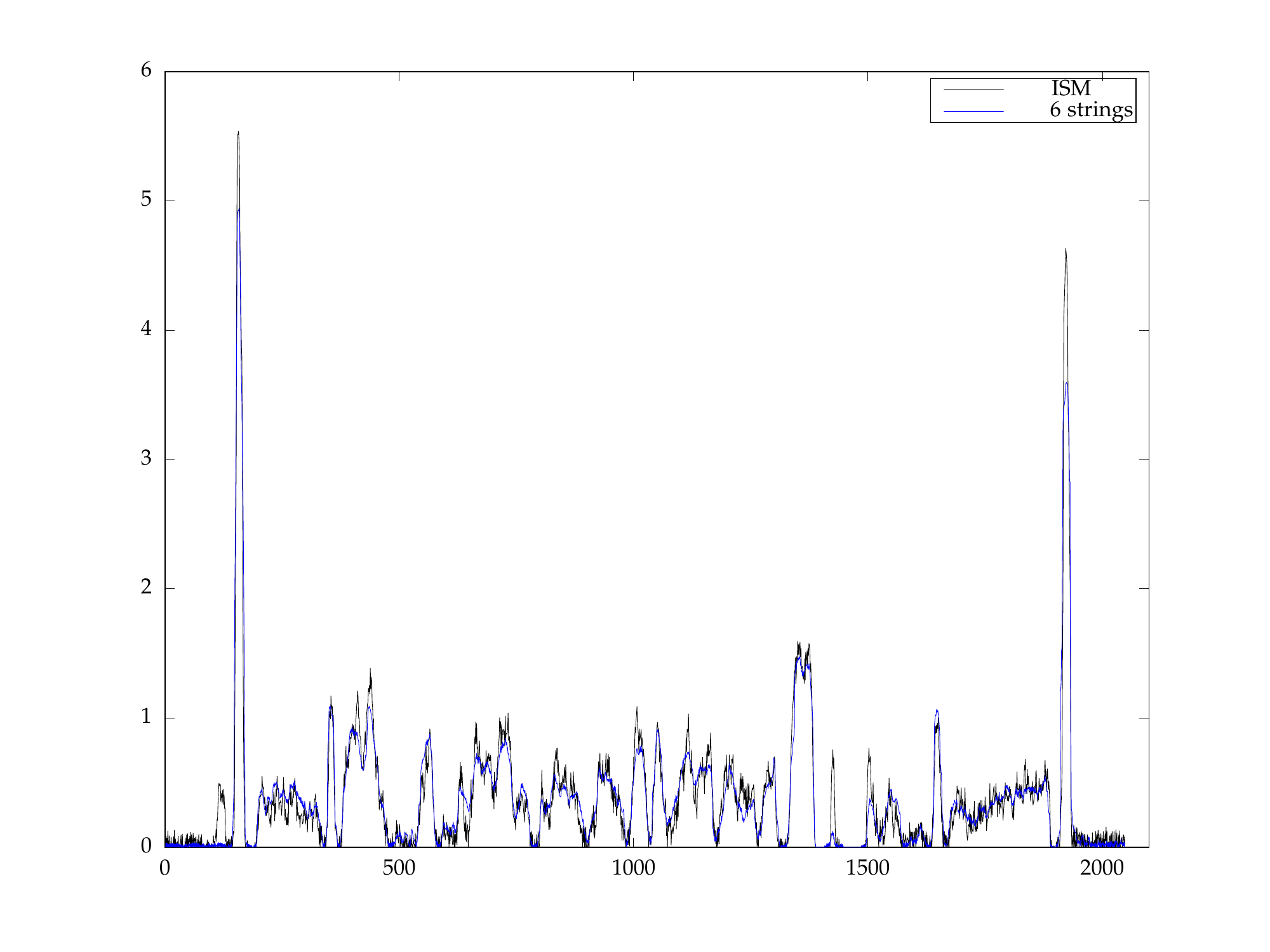}
\caption{Profile lines from images in Figure \ref{ovas_imagens}.} \label{profile}
\end{figure}

\section{Final comments} \label{sec.6}

We have presented a new String-Averaging Incremental Subgradients family of algorithms. 
The theoretical convergence analysis of the method was established and experiments were performed in order to assess the 
effectiveness of the algorithm.
The method featured a good performance in practice, being able to reconstruct images with superior quality when compared to 
classical incremental subgradient algorithms. Furthermore, algorithmic parameters selection was shown to be robust across a 
range of tomographic experiments. The discussed theory involves solving non-smooth constrained convex optimization problems and, 
in this sense, more general models can be numerically addressed by the presented method.
Future work may be related to the application of the string-averaging technique in incremental subgradient algorithms with 
stochastic errors, such as those that appear in~\cite{sundhar2009incremental}~and~\cite{sundhar2010distributed}.

\section*{Acknowledgements}

We would like to thank the LNLS for providing the beam time for the tomographic acquisition, 
obtained under proposal number 17338. We also want to thank Prof. Marcelo dos Anjos (Federal University of Amazonas) for 
kindly providing the fish egg samples used in the presented experimentation. RMO was partially funded by FAPESP Grant 
2015/10171-2. ESH was partially funded by FAPESP Grants 2013/16508-3 and 2013/07375-0 and CNPq Grant 311476/2014-7. EFC 
was funded by FAPESP under Grant 13/19380-8 and CNPq under Grant 311290/2013-2.

\bibliographystyle{acm}
\bibliography{references}

\end{document}